\documentclass[11pt]{amsart}
\usepackage{amsmath,amsfonts,amssymb}

\usepackage{latexsym}
\usepackage{eufrak}
\theoremstyle{plain}
\newtheorem{lemma}{Lemma}[section]
\newtheorem{theorem}[lemma]{Theorem}

\newtheorem{proposition}[lemma]{Proposition}

\theoremstyle{definition}
\newtheorem{remark}[lemma]{Remark}
\newtheorem{definition}[lemma]{Definition}

\numberwithin{equation}{section}

\def\P{\mathbb{P}}
\def\E{\mathbb{E}}

\def\M{\text{M}}

\newcommand{\ZZ}{\mathbb{Z}}

\newcommand{\RR}{\mathbb{R}}

\newcommand{\mA}{\mathcal{A}}

\newcommand{\mC}{\mathcal{C}}
\newcommand{\mD}{\mathcal{D}}
\newcommand{\mE}{\mathcal{E}}

\newcommand{\mJ}{\mathcal{J}}

\newcommand{\mS}{\mathcal{S}}
\newcommand{\mT}{\mathcal{T}}

\newcommand{\mV}{\mathcal{V}}

\newcommand{\bks}{\backslash}
\newcommand{\half}{\frac{1}{2}}

\newcommand{\tZ}{\tilde{Z}}

\newcommand{\tw}{\tilde{\omega}}

\newcommand{\hr}{\hat{r}}
\newcommand{\hs}{\hat{s}}

\newcommand{\hH}{\hat{H}}

\newcommand{\hw}{\hat{\omega}}

\newcommand{\bbL}{\mathbb{L}}

\newcommand{\hmS}{\hat{\mathcal{S}}}

\begin{document}

\title[Subgaussian concentration and rates of convergence.]{Subgaussian concentration and rates of convergence in directed polymers.}
\author{Kenneth S. Alexander}
\address{Department of Mathematics \\
University of Southern California\\
Los Angeles, CA  90089-2532 USA}
\email{alexandr@usc.edu}
\thanks{The research of the first author was supported by NSF grants
DMS-0405915 and DMS-0804934. The research of the second author was partially supported by IRG-246809. The second author would also like to acknowledge the hospitality of the University of Southern California and Academia Sinica, Taipei, where parts of this work were completed.
}
\author{Nikos Zygouras}
\address{Department of Statistics\\
University of Warwick\\
Coventry CV4 7AL, UK}
\email{N.Zygouras@warwick.ac.uk}

\keywords{directed polymers, concentration, modified Poincar\'e inequalities, coarse graining}
\subjclass{Primary: 82B44; Secondary: 82D60, 60K35}

\maketitle

\begin{abstract} We consider directed random polymers in $(d+1)$ dimensions with nearly gamma i.i.d. disorder.  We study the partition function $Z_{N,\omega}$ and establish exponential concentration of $\log Z_{N,\omega}$ about its mean on the subgaussian scale $\sqrt{N/\log N}$ . This is used to show that $\mathbb{E}[ \log Z_{N,\omega}]$ differs from $N$ times the free energy by an amount which is also subgaussian (i.e. $o(\sqrt{N})$), specifically $O( \sqrt{\frac{N}{\log N}}\log \log N)$.
\end{abstract}

\section{Introduction.}
We consider a symmetric simple random walk on $Z^d$, $d\geq 1$. We denote the paths of the walk by $(x_n)_{n\geq 1}$ and its distribution (started from 0) by $P$.  Let $(\omega_{n,x})_{n\in \mathbb{N},x\in \mathbb{Z}^d}$ be a collection of i.i.d. mean-zero random variables with distribution $\nu$ and denote their joint distribution by $\mathbb{P}$.  We think of $(\omega_{n,x})_{n\in \mathbb{N},x\in \mathbb{Z}^d}$ as a random potential with the random walk moving inside this potential. This interaction gives rise to the directed polymer in a random environment and can be formalised by the introduction of the following Gibbs measure on paths of length $N$:
\begin{eqnarray*}
d\mu_{N,\omega}=\frac{1}{Z_{N,\omega}} e^{\beta\sum_{n=1}^N\omega_{n,x_n}} dP,
\end{eqnarray*}
where $\beta>0$ is the inverse temperature. The normalisation 
\begin{eqnarray}\label{partitionf}
Z_{N,\omega}=E\left[\exp\left(\beta\sum_{n=1}^N\omega_{n,x_n}\right)\right] 
\end{eqnarray}
 is the partition function. 
 
 A central question for such polymers is how the fluctuations of the path are influenced by the presence of the disorder. Loosely speaking, consider the two exponents $\xi$ and $\chi$ given by
 \begin{eqnarray*}
 E_{N,\omega}[|x_N|^2]\sim N^{2\xi}, \quad  \mathbb{V}\text{ar}\left(\log Z_{N,\omega}\right)\sim N^{2\chi}.
 \end{eqnarray*}
It is believed that $\chi<1/2$ for all $\beta > 0 $ and all $d$ (see \cite{KS}.)
It is expected and partially confirmed for some related models (\cite{Wut1}, \cite{Chat1}) that the two exponents $\chi,\xi$ are related via 
  \begin{eqnarray}\label{hyperscaling}
 \chi=2\xi-1.
 \end{eqnarray}
 So there is reason for interest in the fluctuations of $\log Z_{N,\omega}$, and in particular in establishing that these fluctuations are \emph{subgaussian}, that is, $o(N^{1/2})$, as compared to the gaussian scale $N^{1/2}$.  It is the $o(\cdot)$ aspect that has not previously been proved:  in \cite{Pi97} it is proved that in the point-to-point case (that is, with paths $(x_n)_{n\geq 1}$ restricted to end at a specific site at distance $N$ from the origin) one has variance which is $O(N)$ when the disorder has finite variance, and an exponential bound for $|\log Z_{N,\omega} - \E\log Z_{N,\omega}|$ on scale $N^{1/2}$ when the disorder has an exponential moment.
 
 The zero-temperature case of the polymer model is effectively last passage percolation.  More complete results exist in this case in dimension $1+1$, for specific distributions \cite{Joh}. There, based on exact computations related to combinatorics and random matrix theory, not only the scaling exponent $\chi$ for the directed last passage time  was obtained, but also its limiting distribution after centering and scaling. A first step towards an extension of this type of result in the case of directed polymers in dimension $1+1$ for particular disorder is made in \cite{COSZ}; see also \cite{BC} for a step towards asymptotics.  The best known result for undirected point-to-point last passage percolation is in \cite{BKS03}, stating that for $v\in \mathbb{Z}^d$, $d\geq 2$, one has $\mathbb{V}\text{ar}(\max_{\gamma:0\to v}\sum_{x\in\gamma}\omega_x)\leq C|v|/\log|v|$,  when the disorder $\omega$ is Bernoulli. Some results on sublinear  variance estimates for directed last passage percolation in $1+1$ dimensions with gaussian disorder were obtained in \cite{Chat2}, but the type of estimates there does not extend to higher dimensions, or to directed polymers at positive temperature. The assumption of gaussian disorder is also strongly used there. In \cite{G} estimates of the variance of directed last passage percolation are obtained via a coupling method, which appears difficult to extend to the case of polymers. In \cite{BR08} exponential concentration estimates on the scale $(|v|/\log |v|)^{1/2}$ were obtained for first passage percolation, for a large class of disorders.
 
  The extension of these results to directed polymers is not straightforward. This can be be seen, for example, from the fact that subgaussian fluctuations for a point-to-point directed polymer can naturally fail.   Such failure occurs, for example, if one restricts the end point of a $(1+1)$-dimensional directed polymer to be $(N,N)$. Then \eqref{partitionf} reduces to a sum of i.i.d. variables whose fluctuations are therefore gaussian. 
  
  The first result of the present paper is to obtain exponential concentration estimates on the scale $(N/\log N)^{1/2}$. Specifically, for {\it nearly gamma} disorder distributions (see Definition \ref{nearly gamma}, a modification of the definition in \cite{BR08}) we prove the following; here and throughout the paper we use $K_i$ to denote constants which depend only on $\beta$ and $\nu$.

\begin{theorem} \label{expconc}
Suppose the disorder distribution $\nu$ is nearly gamma with $\int e^{4\beta|\omega|} \nu(d\omega)<\infty$. Then there exist $K_0, K_1$ such that
\begin{eqnarray*}
\mathbb{P}\left( \left| \log Z_{N,\omega}-\mathbb{E}\log Z_{N,\omega}\right|>t\sqrt{\frac{N}{\log N}}\right)\leq K_0e^{-K_1t},
\end{eqnarray*}
for all $N\geq 2$ and $t>0$.
\end{theorem}

The nearly gamma condition ensures that $\nu$ has some exponential moment (see Lemma \ref{neargmoment}), so for small $\beta$ the exponential moment hypothesis in Theorem \ref{expconc} is redundant.
The proof follows the rough outline of \cite{BR08}, and uses some results from there, which we summarize in Section \ref{review conc}. 
 

We use Theorem \ref{expconc}, in combination with coarse graining techniques motivated by \cite{Al11}, to provide subgaussian estimates of the rate of convergence of $N^{-1}\mathbb{E}\log Z_{N,\omega}$ to the free energy. 
Here the \emph{free energy} of the polymer (also called the \emph{pressure}) is defined as
\begin{equation} \label{aslimit}
p(\beta)=  \lim_{N\to\infty} \frac{1}{N}\log Z_{N,\omega}  \quad \mathbb{P}-\text{a.s.}
\end{equation}
The existence of the free energy is obtained by standard subadditivity arguments and concentration results \cite{CSY03}, which furthermore guarantee that 
\begin{eqnarray} \label{pbeta}
p(\beta)&=& \lim_{N\to\infty} \frac{1}{N}\E\log Z_{N,\omega}\\
&=&\sup_N \frac{1}{N}\E\log Z_{N,\omega}.
\end{eqnarray}
Specifically, our second main result is as follows.

\begin{theorem}\label{ratemain}
  Under the same assumptions as in Theorem \ref{expconc}, there exists $K_2$ such that for all $N\geq 3$,
  \begin{equation} \label{rate2}
  Np(\beta) \geq \E\log Z_{N,\omega} \geq Np(\beta) - K_2N^{1/2}\frac{\log \log N}{(\log N)^{1/2}}.
  \end{equation}
  \end{theorem}
  
Controlling the speed of convergence of the mean is useful when one considers deviations of $N^{-1}\log Z_{N,\omega}$ from its limit $p(\beta)$ instead of from its mean, analogously to \cite{Chat1}.

Regarding the organization of the paper, in Section \ref{review conc}
we review certain concentration inequalities and related results, mostly from \cite{BR08}, and give an extension of the definition from \cite{BR08} of a nearly gamma distribution so as to allow non-positive variables. In Section \ref{proofcon} we provide the proof of Theorem \ref{expconc}. In Section \ref{proofrates} we provide the proof of Theorem \ref{ratemain}.  Finally, in Section \ref{prooflem} we provide the proof of a technical lemma used in Section \ref{proofrates}.

\section{Preliminary Results on Concentration and Nearly Gamma Distributions.}\label{review conc}

Let us first define the class of nearly gamma distributions. This class, introduced in \cite{BR08} is quite wide and in particular it includes the cases of Gamma and normal variables. The definition given in \cite{BR08} required that the support does not include negative values. Here we will extend this definition in order to accommodate such values as well.

\begin{definition}\label{nearly gamma}
Let $\nu$ be a probability measure on $\mathbb{R}$, absolutely continuous with respect to the Lebesque measure, with density $h$ and cumulative distribution function $H$. Let also $\Phi$ be the cumulative distribution function of the standard normal. $\nu$ is said to be {\it nearly gamma} (with parameters $A,B$) if
\begin{itemize}
\item[(i)] The support $I$ of $\nu$ is an interval.\\
\item[(ii)] $h(\cdot)$ is continuous on $I$.\\
\item[(iii)] For every $y\in I$ we have
\begin{eqnarray} \label{psidef}
  \psi(y) := \frac{\Phi'\circ\Phi^{-1} (H(y))}{h(y)}\leq \sqrt{B+A|y|},
\end{eqnarray}
where $A,B$ are nonnegative constants.
\end{itemize}
\end{definition}

The motivation for this definition (see \cite{BR08}) is that $H^{-1}\circ\Phi$ maps a gaussian variable to one with distribution $\nu$, and $\psi(y)$ is the derivative of this map, evaluated at the inverse image of $y$.  With the bound on $\psi$ in (iii), the log Sobolev inequality satisfied by a gaussian distribution with respect to the differentiation operator translates into a useful log Sobolev inequality satisfied by the distribution $\nu$  with respect to the operator $\psi(y)d/dy$.

It was established in \cite{BR08} that a distribution is nearly gamma if (i), (ii) of Definition \ref{nearly gamma} are valid, and 
(iii) is replaced by
\begin{itemize}
\item[(iv)]
 if $ I=[\nu_-,\nu_+]$ with $ |\nu_{\pm}|<\infty$, then 
\begin{eqnarray*}
\frac{h(x)}{|x-\nu_\pm|^{\alpha_{\pm}}},
\end{eqnarray*}
remains bounded away from zero and infinity for $x\sim \nu_{\pm}$, for some $\alpha_\pm>-1$.
\item[(v)]
If $\nu_+=+\infty$ then
\begin{eqnarray*}
\frac{\int_{x}^\infty h(t) dt}{h(x)}
\end{eqnarray*}
remains bounded away from zero and infinity, as $x\to+\infty$. The analogous statement is valid if $\nu_-=-\infty$.
\end{itemize}

The nearly gamma property ensures the existence of an exponential moment, as follows.  

\begin{lemma}\label{neargmoment}
Suppose the distribution $\nu$ is nearly gamma with parameters $A,B$.  Then $\int e^{tx}\ \nu(dx) < \infty$ for all $t<2/A$.
\end{lemma}
\begin{proof}
Let $T=H^{-1}\circ\Phi$, so that $T(\xi)$ has distribution $\nu$ for standard normal $\xi$; then \eqref{psidef} is equivalent to 
\[
  T'(x) \leq \sqrt{B+A|T(x)|} \quad \text{for all } x \in \RR.
\]
Considering $T(x)\geq 0$ and $T(x)<0$ separately, it follows readily from this that 
\[
  \left| \frac{d}{dx} \sqrt{(B+A|T(x)|)} \right| \leq \frac{A}{2} \quad\text{for all $x$ with } T(x)\neq 0,
\]
so for some constant $C$ we have $\sqrt{(B+A|T(x)|)} \leq C+A|x|/2$, or 
\[
  |T(x)| \leq \frac{C^2-B}{A} + C|x| + \frac{A}{4}x^2,
\]
and the lemma follows.
\end{proof}

For $\omega \in \RR^{\ZZ^{d+1}}$ and $(m,y) \in \ZZ^{d+1}$ we define $\hw^{(m,y)} \in \RR^{\ZZ^{d+1}\bks \{(m,y)\}}$ by the relation $\omega = (\hw^{(m,y)},\omega_{m,y})$.  In other words, $\hw^{(m,y)}$ is $\omega$ with the coordinate $\omega^{(m,y)}$ removed.  Given a function $F$ on $\RR^{\ZZ^{d+1}}$ and a configuration $\omega$, the average sensitivity of $F$ to changes in the $(m,y)$ coordinate is given by
\[
  Y^{(m,y)}(\omega) := \int \left| F(\hw^{(m,y)},\tw_{m,y}) - F(\omega) \right|\ d\mathbb{P}(\tw_{m,y}).
\]
We define
\[
  Y_N(\omega) := \sum_{(m,y) \in \{1,\dots,N\}\times\ZZ^d} Y^{(m,y)}(\omega),
\]
\[
  \rho_N := \sup_{(m,y) \in \{1,\dots,N\}\times\ZZ^d} \sqrt{ \mathbb{E}\left[ (Y^{(m,y)})^2 \right]},
\]
\[
  \sigma_N := \sqrt{ \mathbb{E}\left( Y_N^2\right) }.
\]
We use the same notation (a mild abuse) when $F$ depends on only a subset of the coordinates.

We are now ready to state the theorem of Benaim and Rossignol \cite{BR08}, specialized to the operator $\psi(s)d/ds$ applied to functions $e^{\frac{\theta}{2}F(\hw_{m,y},\cdot)}$.

\begin{theorem}\label{main concentration}
Let  $F\in L^2(\nu^{\{1,\dots,N\}\times\ZZ^d})$ and let $\rho_N,\sigma_N$ be as above.  Suppose that there exists $K>e\rho_N\sigma_N$ such that
\begin{eqnarray} \label{assumption Poincare}
\sum_{(m,y) \in \{1,\dots,N\}\times\ZZ^d} \mathbb{E} \left[ \left( \psi(\omega_{m,y}) \frac{\partial}{\partial\omega_{m,y}}
  e^{\frac{\theta}{2}F}\right)^2\right] 
\leq
K\theta^2 \mathbb{E}\left[e^{\theta F}\right]
\end{eqnarray}
for all $|\theta|<\frac{1}{2\sqrt{l(K)}}$ where 
\begin{eqnarray*}
l(K)=\frac{K}{\log \frac{K}{\rho_N\sigma_N \log \frac{K}{\rho_N\sigma_N } }}.
\end{eqnarray*}
Then for every $t>0$ we have that
\begin{eqnarray*}
\mu\left( |F-\mathbb{E}[F] |\geq t\sqrt{l(K)} \right) \leq 8e^{-t}.
\end{eqnarray*}
\end{theorem}

Observe that if $K$ is of order $N$, then a bound on $\rho_N\sigma_N$ of order $N^\alpha$ with $\alpha<1$ is sufficient to ensure that $l(K)$ is of order $N/\log N$.  In particular it is sufficient to have $\sigma_N$ of order $N$ and $\rho_N$ of order $N^{-\tau}$ with $\tau>0$, which is what we will use below.

\section{Concentration for the Directed Polymer.}\label{proofcon}

In this section we will establish the first main result of the paper, which is Theorem \ref{expconc}. We assume throughout that the distribution $\nu$ of the disorder is nearly gamma with parameters $A,B$. We finally denote $\mathbb{P}=\nu^{\mathbb{Z}^{d+1}}$.  We write $\mu(f)$ for the integral of a function $f$ with respect to a measure $\mu$.

Let $(n,x)\in \mathbb{N}\times \mathbb{Z}^d$.  We denote the partition function of the directed polymer of length $N$ in the shifted environment $\omega_{n+\cdot,x+\cdot}$ by
\begin{eqnarray}\label{ZN}
Z_{N,\omega}^{(n,x)}:=E\left[ e^{\beta\sum_{i=1}^N\omega_{n+i,x+x_i}} \right],
\end{eqnarray}
and let $\mu_{N,\omega}^{(n,x)}$ be the corresponding Gibbs measure. For $I\subset \mathbb{N}\times \mathbb{Z}^d$ we define
\begin{eqnarray*}
\overline{F}_{N,\omega}^I:=\frac{1}{| I |}\sum_{(n,x)\in I} \log Z_{N,\omega}^{(n,x)}.
\end{eqnarray*}
Define the set of paths from the origin
\[
  \Gamma_N = \{ \{(i,x_i)\}_{i\leq N}:  x_0=0,|x_i-x_{i-1}|_1=1 \text{ for all } i\};
\]
we write $\gamma_N=\{(i,x_i)\colon i=0,\dots,N\}$ for a generic or random polymer path in $\Gamma_N$.  Let
\begin{eqnarray}\label{absmax}
\mathcal{M}_{N,\omega}=\max_{\gamma_N} \sum_{(m,y)\in\gamma_N}|\omega_{m,y}|,
\end{eqnarray}
and let $\mathcal{M}^{(n,x)}_{N,\omega}$ denote the same quantity for the shifted disorder, analogously to \eqref{ZN}.

\begin{proposition} \label{Poinc}
There exists $\theta_0(\beta,\nu)$ such that for all $|\theta|<\theta_0$ and $|I|\leq (2d)^N$, the function $\overline{F}_{N,\omega}^I$ satisfies the following Poincar\'e type inequality:
\begin{eqnarray*}
\sum_{(m,y)\in \mathbb{N}\times\mathbb{Z}^d} \mathbb{E}\left[\left( \psi(\omega_{m,y})\frac{\partial }{\partial \omega_{m,y}} e^{\frac{\theta}{2} \overline{F}_{N,\omega}^I } \right)^2\right]
\leq 
C_{AB} \theta^2\beta^2 N\,\,\mathbb{E}\left[ e^{\theta\overline{F}_{N,\omega}^I} \right],
\end{eqnarray*}
where $C_{AB}$ is a constant depending on the nearly gamma parameters $A,B$.
\end{proposition}

\begin{proof}
By the definition of nearly gamma we have that
\begin{align} \label{ABsub}
\sum_{(m,y)\in \mathbb{N}\times\mathbb{Z}^d} &\mathbb{E}\left[\left( \psi(\omega_{m,y})\frac{\partial }{\partial \omega_{m,y}} e^{\frac{\theta}{2} \overline{F}_{N,\omega}^I } \right)^2\right] \notag \\
&\leq 
B\mathbb{E}\left[ \left(\frac{\partial}{\partial\omega_{m,y}} e^{\frac{\theta}{2} \overline{F}_{N,\omega}^I }  \right)^2 \right] +
A \mathbb{E}\left[ |\omega_{m,y}| \left(\frac{\partial}{\partial\omega_{m,y}} e^{\frac{\theta}{2} \overline{F}_{N,\omega}^I }  \right)^2 \right].
\end{align}
Regarding the first term on the right side of \eqref{ABsub}, we have
\begin{eqnarray*}
\frac{\partial \overline{F}^I_{N,\omega}}{\partial \omega_{m,y}}=
 \frac{\beta}{| I |}\sum_{(n,x)\in I} \mu_{N,\omega}^{n,x}(1_{(m-n,y-x)\in \gamma_N})
\end{eqnarray*}
and
\begin{align}\label{Poincare estimate}
\sum_{(m,y)\in \mathbb{N}\times\mathbb{Z}^d} &\mathbb{E}\left[\left( \frac{\partial}{\partial\omega_{m,y}}e^{\frac{\theta}{2}\overline{F}_{N,\omega}^I} \right)^2\right]\nonumber\\
&=\frac{1}{4}\theta^2
\sum_{(m,y)\in \mathbb{N}\times\mathbb{Z}^d} \mathbb{E}\left[\left( \frac{\partial \overline{F}_{N,\omega}^I }{\partial\omega_{m,y}} 
                                                                           \right)^2
                                                                    e^{\theta\overline{F}_{N,\omega}^I}
                                                                    \right]\nonumber\\
&=  \frac{1}{4}\theta^2 \beta^2 \sum_{(m,y)\in \mathbb{N}\times\mathbb{Z}^d} \mathbb{E}\left[\left( 
                                                                             \frac{1}{| I |}\sum_{(n,x)\in I} \mu_{N,\omega}^{n,x}(1_{(m-n,y-x)\in \gamma_N})
                                                                           \right)^2
                                                                    e^{\theta\overline{F}_{N,\omega}^I}
                                                                    \right]\nonumber\\          
 &\leq  \frac{1}{4}\theta^2\beta^2 \sum_{(m,y)\in \mathbb{N}\times\mathbb{Z}^d} \mathbb{E}\left[ 
                                                                             \frac{1}{| I |}\sum_{(n,x)\in I} \mu_{N,\omega}^{n,x}(1_{(m-n,y-x)\in \gamma_N})
                                                                    \,\,e^{\theta\overline{F}_{N,\omega}^I}
                                                                   \right]\\        
  &=    \frac{1}{4}\theta^2\beta^2  N\,\, \mathbb{E}\left[   
                                                                                          e^{\theta\overline{F}_{N,\omega}^I}
                                                                                   \right]  ,\nonumber                                                                                                                                                                   
\end{align}
where the last equality is achieved by performing first the summation over $(m,y)$ and using that the range of the path consists of $N$ sites after the starting site. Regarding the second term on the right side of \eqref{ABsub}, we define $\mathcal{M}^{I}_{N,\omega}=\max_{(n,x)\in I} \mathcal{M}^{(n,x)}_{N,\omega}$ for a set $I\subset \mathbb{N}\times \mathbb{Z}^{d}$. We then have
 $-\beta \mathcal{M}^{I}_{N,\omega} \leq \overline{F}_{N,\omega}^I \leq \beta \mathcal{M}^{I}_{N,\omega} $ so following similar steps as in \eqref{Poincare estimate} we have
 \begin{eqnarray}\label{termtwo}
  \sum_{(m,y)\in \mathbb{N}\times\mathbb{Z}^d} 
   && \mathbb{E}\left[|\omega_{m,y}| \left(\frac{\partial}{\partial\omega_{m,y}} e^{\frac{\theta}{2}   \overline{F}_{N,\omega}^I }  \right)^2 \right]\nonumber\\
&\leq&
      \frac{1}{4}\theta^2\beta^2 \sum_{(m,y)\in \mathbb{N}\times\mathbb{Z}^d} \mathbb{E}\left[ 
                                                                             \frac{1}{| I |}\sum_{(n,x)\in I} \mu_{N,\omega}^{n,x}
                                                                             (|\omega_{m,y}|1_{(m-n,y-x)\in \gamma_N})
                                                                    \,\,e^{\theta\overline{F}_{N,\omega}^I}
                                                                 \right] \nonumber\\        
&\leq&  \frac{1}{4}\theta^2\beta^2
              \mathbb{E}\left[  \mathcal{M}^{I}_{N,\omega}  e^{\theta \overline{F}_{N,\omega}^I}
                          \right]    \nonumber  \\
 &\leq&  \frac{1}{4}\theta^2\beta^2 \Big( bN
              \mathbb{E}\left[  e^{\theta \overline{F}_{N,\omega}^I}
                          \right] +
               \mathbb{E}\left[ \mathcal{M}^{I}_{N,\omega} e^{|\theta| \beta \mathcal{M}^{I}_{N,\omega}}
                         ; \mathcal{M}^{I}_{N,\omega} > bN \right]                   
                          \Big),                                                                    
 \end{eqnarray}
where $b$ a constant to be specified.  We would like to show that the second term on the right side of \eqref{termtwo} is smaller that the first one. First, in the case that $\theta>0$, since the disorder has mean zero, bounding $Z_{N,\omega}^{(n,x)}$ below by the contribution from any one path and then applying Jensen's inequality to the expectation $\mathbb{E}[\cdot]$ we obtain
 \begin{eqnarray}\label{viaJensen}
 \mathbb{E}\left[ e^{\theta \overline{F}_{N,\omega}^I}\right] \geq  e^{-\theta N \log(2d)},
 \end{eqnarray}
while  in the case that $\theta<0$, applying Jensen's inequality to the average over $I$ gives
\[
  \mathbb{E}\left[ e^{\overline{F}_{N,\omega}^I}\right] \leq \mathbb{E}\left[ Z_{N,\omega} \right] = e^{\lambda(\beta)N},
\]
with $\lambda(\beta)$ the log-moment generating function of $\omega$, and hence, taking the $\theta$ power and then applying Jensen's inequality to $\mathbb{E}[\cdot]$,
  \begin{eqnarray}\label{viaJensen2}
 \mathbb{E}\left[ e^{\theta \overline{F}_{N,\omega}^I}\right] \geq e^{\theta N \lambda(\beta)}.
 \end{eqnarray}
 Moreover, for $b>0$ we have 
 \begin{eqnarray}\label{bterm}
 &&\mathbb{E}\left[ \mathcal{M}^{I}_{N,\omega} e^{|\theta| \beta \mathcal{M}^{I}_{N,\omega}}
                         ; \mathcal{M}^{I}_{N,\omega} > bN \right] \nonumber \\
  &=& bN e^{|\theta| \beta bN} \mathbb{P}(\mathcal{M}^{I}_{N,\omega}>bN)    
     + N\int_b^\infty (1+|\theta| \beta uN) e^{|\theta| \beta uN} \mathbb{P}(\mathcal{M}^{I}_{N,\omega}>uN) du.             
 \end{eqnarray}
 Denoting by $\mathcal{J}(\cdot)$ the large deviation rate function related to $|\omega|$ we have that \eqref{bterm} is bounded by
 \begin{eqnarray}\label{bbterm}
   bN(2d)^N |I|  e^{(|\theta|\beta b-\mathcal{J}(b))N} + N(2d)^N |I| \int_b^\infty (1+|\theta|\beta u N) e^{(|\theta|\beta u-\mathcal{J}(u))N} du.
 \end{eqnarray}
Let $0<L< \lim_{x\to \infty} \mathcal{J}(x)/x$ (which exists since $\mathcal{J}(x)/x$ is nondecreasing for $x>\E|\omega|$) and choose $b$ large enough so $\mathcal{J}(b)/b>L$.  Then provided $|\theta|$ is small enough (depending on $\beta, \nu$) and $b$ is large enough (depending on $\nu$), \eqref{bbterm}is bounded above by
 \begin{eqnarray*}
 &&bN(2d)^{2N} e^{(|\theta|\beta-L)bN}+ N(2d)^{2N}\int_b^\infty (1+|\theta| \beta uN) e^{(|\theta|\beta-L)uN}du\\
 &\leq& bN (2d)^{2N} e^{-\frac{L}{2}bN} +N (2d)^{2N}\int_b^\infty (1+|\theta| \beta uN) e^{-\frac{L}{2}uN}du\\
 &\leq&e^{-LbN/4}\\
 &\leq& \mathbb{E}\left[ e^{\theta \overline{F}^I_{N,\omega}}\right],
 \end{eqnarray*} 
 where the last inequality uses \eqref{viaJensen} and \eqref{viaJensen2}. This combined with \eqref{termtwo} and \eqref{Poincare estimate} completes the proof.
\end{proof}

The averaging over sets $I$ used in the preceding proof is related to the auxiliary randomness used in the main proof in \cite{BKS03}.

Define the point-to-point partition function
\[
  Z_{N,\omega}(z)=E\left[\exp\left(\beta\sum_{n=1}^N\omega_{n,x_n}\right)1_{x_N=z}\right] 
\]
and let $\mu_{N,\omega,z}$ be the corresponding Gibbs measure.
With $I$ fixed, we define
\begin{eqnarray*}
W_{N,\omega}^{(m,y)}:=\int\left|\overline{F}_{N,(\hat\omega^{(m,y)},\tilde\omega_{m,y})}^I- \overline{F}_{N,\omega}^I \right|d\mathbb{P}(\tilde\omega_{m,y}),
\end{eqnarray*}
\[
  L_{N,\omega}^{(m,y)}(z):=\int\left| \log Z_{N,(\hat\omega^{(m,y)},\tilde\omega_{m,y})}(z) - \log Z_{N,\omega}(z) \right|
    d\mathbb{P}(\tilde\omega_{m,y}),
\]
\begin{eqnarray*}
W_{N,\omega}:= \sum_{(m,y)\in \mathbb{N}\times \mathbb{Z}^d} W_{N,\omega}^{(m,y)}, \qquad
  L_{N,\omega}(z):= \sum_{(m,y)\in \mathbb{N}\times \mathbb{Z}^d} L_{N,\omega}^{(m,y)}(z),
\end{eqnarray*}
\begin{eqnarray*}
W_{N,\omega,\pm}^{(m,y)}:=\int\left(\overline{F}_{N,(\hat\omega^{(m,y)},\tilde\omega_{m,y})}^I- \overline{F}_{N,\omega}^I \right)_\pm d\mathbb{P}(\tilde\omega_{m,y}),
\end{eqnarray*}
\[
  L_{N,\omega,\pm}^{(m,y)}(z):=\int\left( \log Z_{N,(\hat\omega^{(m,y)},\tilde\omega_{m,y})}(z) 
    - \log Z_{N,\omega}(z) \right)_\pm d\mathbb{P}(\tilde\omega_{m,y}),
\]
and
\begin{eqnarray*}
W_{N,\omega,\pm}:= \sum_{(m,y)\in \mathbb{N}\times \mathbb{Z}^d} W_{N,\omega,\pm}^{(m,y)}, \qquad
  L_{N,\omega,\pm}(z):= \sum_{(m,y)\in \mathbb{N}\times \mathbb{Z}^d} L_{N,\omega,\pm}^{(m,y)}(z).
\end{eqnarray*}
We finally define 
\begin{eqnarray*}
r_N:=\sup_{(m,y)\in \mathbb{N}\times \mathbb{Z}^d} \sqrt{\mathbb{E}\left[\big(W_{N,\omega}^{(m,y)} \big)^2\right] }, \qquad
  \hr_N(z) := \sup_{(m,y)\in \mathbb{N}\times \mathbb{Z}^d} \sqrt{\mathbb{E}\left[\big(L_{N,\omega}^{(m,y)}(z) \big)^2\right] }, 
\end{eqnarray*}
\begin{eqnarray*}
s_N:=\sqrt{\mathbb{E}\left[\big(W_{N,\omega} \big)^2\right] }, \qquad
  \hs_N(z):=\sqrt{\mathbb{E}\left[\big(L_{N,\omega}(z) \big)^2\right] },
\end{eqnarray*}
\begin{eqnarray*}
r_N^{\pm}:=\sup_{(m,y)\in \mathbb{N}\times \mathbb{Z}^d} \sqrt{\mathbb{E}\left[\big(W_{N,\omega,\pm}^{(m,y)} \big)^2\right] },
  \qquad \hr_N^{\pm}(z):=\sup_{(m,y)\in \mathbb{N}\times \mathbb{Z}^d} 
  \sqrt{\mathbb{E}\left[\big(L_{N,\omega,\pm}^{(m,y)}(z) \big)^2\right] }
\end{eqnarray*}
and 
\begin{eqnarray*}
s_N^\pm:=\sqrt{\mathbb{E}\left[\big(W_{N,\omega,\pm} \big)^2\right] }, \qquad
  \hs_N^\pm(z):=\sqrt{\mathbb{E}\left[\big(L_{N,\omega,\pm}(z) \big)^2\right] }.
\end{eqnarray*}
It is clear that $r_n\leq r_N^++r_N^-$ and $s_n\leq s_N^++s_N^-$.

We make use of two choices of the set $I$ of sites:  let $0<\alpha<1/2$ and 
\[
  I_\pm^\alpha:=\{(n,x)\in \mathbb{N}\times \mathbb{Z}^d\colon n=\pm N^{\alpha},|x|_\infty<N^\alpha\}.
\]

\begin{proposition}\label{rNsN}
For $\alpha<1/2$ and $I = I_\pm^\alpha$, there exists $K_3$ such that the following estimates hold true:
\[
r_N^\pm \leq \frac{1}{|I_+^\alpha|^{1/4}} K_3, \qquad
  \hr_N^\pm(z) \leq K_3,
\]
\[
s_N^\pm \leq K_3 N, \qquad \hs_N^\pm(z) \leq K_3 N.
\]
\end{proposition}
\begin{proof}
We first consider $r_N^\pm$ and $s_N^\pm$.  Observe that
\begin{eqnarray} \label{posnegpart}
\left(\overline{F}_{N,(\hat\omega^{(m,y)},\tilde\omega_{m,y})}^{I_+^\alpha}- \overline{F}_{N,\omega}^{I_+^\alpha} \right)_\pm
   &=& \left( \frac{1}{| I_+^\alpha |}\sum_{(n,x)\in I_+^\alpha} \big( 
                                                                                      \log Z _{N,(\hat\omega^{(m,y)},\tilde\omega_{m,y})}^{(n,x)}- \log Z_{N,\omega}^{(n,x)}
                                                                                      \big)
      \right)_\pm\nonumber\\
      &\leq&  \frac{1}{| I_+^\alpha |}\sum_{(n,x)\in I_+^\alpha} \big( 
                                                                                      \log Z _{N,(\hat\omega^{(m,y)},\tilde\omega_{m,y})}^{(n,x)}- \log Z_{N,\omega}^{(n,x)}
                                                                                      \big)_\pm.
\end{eqnarray}
The difference on the right side can be written as
\begin{eqnarray} \label{logdiff2}
 \log Z _{N,(\hat\omega^{(m,y)},\tilde\omega_{m,y})}^{(n,x)}- \log Z_{N,\omega}^{(n,x)}
 &=&
 \log \frac{1}{Z_{N,\omega}^{(n,x)}}
         E\left[ e^{\beta\sum_{i=1}^N\omega_{n+i,x+x_i}} e^{\beta\big(\tilde\omega_{m,y}-\omega_{m,y}\big)1_{x+x_{m-n}=y}}
 \right]\nonumber\\
 &=&  \log \mu_{N,\omega}^{(n,x)}\left( e^{\beta\big(\tilde\omega_{m,y}-\omega_{m,y}\big)1_{x+x_{m-n}=y}}
 \right)\nonumber\\
 &=& \log\left( 1+ \mu_{N,\omega}^{(n,x)}\left( e^{\beta\big(\tilde\omega_{m,y}-\omega_{m,y}\big)1_{x+x_{m-n}=y}}-1\right)
 \right)\\
 &\leq& \log\left( 1+ e^{\beta\big(\tilde\omega_{m,y}-\omega_{m,y}\big)} \mu_{N,\omega}^{(n,x)}\left( 1_{x+x_{m-n}=y}
    \right) \right) \notag \\
  &\leq& e^{\beta\big(\tilde\omega_{m,y}-\omega_{m,y}\big)} \mu_{N,\omega}^{(n,x)}\left( 1_{x+x_{m-n}=y} \right),\nonumber
\end{eqnarray}
so
\begin{align}\label{plus}
W_{N,\omega,+}^{(m,y)} &= \int \left(\overline{F}_{N,(\hat\omega^{(m,y)},\tilde\omega_{m,y})}^{I_+^\alpha}- \overline{F}_{N,\omega}^{I_+^\alpha} \right)_+\,d\mathbb{P}(\tilde\omega_{m,y})\notag\\
 &\leq\frac{1}{| I_+^\alpha |}\sum_{(n,x)\in I_+^\alpha} \mu_{N,\omega}^{(n,x)}\left( 1_{x+x_{m-n}=y}
 \right)  \int_{\tilde\omega_{m,y}\geq\omega_{m,y}} 
  e^{\beta\big(\tilde\omega_{m,y}-\omega_{m,y}\big)} \,d\mathbb{P}(\tilde\omega_{m,y}).
\end{align}
To bound $r_N^+$, we have using \eqref{plus}:
\begin{align} \label{rNplus}
\mathbb{E}\left[\big(W_{N,\omega,+}^{(m,y)} \big)^2\right] &\leq
  \mathbb{E}\left[ \left( \frac{1}{| I_+^\alpha |} \sum_{(n,x)\in I_+^\alpha}  \mu_{N,\omega}^{(n,x)}\left( 1_{x+x_{m-n}=y}
   \right)  \int_{\tilde\omega_{m,y}\geq\omega_{m,y}} 
  e^{\beta\big(\tilde\omega_{m,y}-\omega_{m,y}\big)} \,d\mathbb{P}(\tilde\omega_{m,y})  \right)^2\right]
  \nonumber \\
  &\leq
  \mathbb{E}\left[ \left( \frac{1}{| I_+^\alpha |} \sum_{(n,x)\in I_+^\alpha}  \mu_{N,\omega}^{(n,x)}\left( 1_{x+x_{m-n}=y}
 \right)\right)^4\right]^{1/2}\nonumber\\
 &\qquad\times
    \mathbb{E}\left[ \left(
   \int_{\tilde\omega_{m,y}\geq\omega_{m,y}} 
  e^{\beta\big(\tilde\omega_{m,y}-\omega_{m,y}\big)} \,d\mathbb{P}(\tilde\omega_{m,y})  \right)^4\right]^{1/2} \nonumber\\
  &\leq
  \mathbb{E}\left[  \frac{1}{| I_+^\alpha |} \sum_{(n,x)\in I_+^\alpha}  \mu_{N,\omega}^{(n,x)}\left( 1_{x+x_{m-n}=y}
 \right)\right]^{1/2} e^{\frac{1}{2}(\lambda(-4\beta)+4\lambda(\beta))} \nonumber\\
  &=
   \mathbb{E}\left[  \frac{1}{| I_+^\alpha |} \sum_{(n,x)\in I_+^\alpha}  \mu_{N,\omega}\left( 1_{x+x_{m-n}=y}
 \right)\right]^{1/2} e^{\frac{1}{2}(\lambda(-4\beta)+4\lambda(\beta))} \\
   &\leq
    \frac{1}{| I_+^\alpha |^{1/2}} e^{\frac{1}{2}(\lambda(-4\beta)+4\lambda(\beta))},  \nonumber
\end{align}
where in the equality we used the homogeneity of the environment and in the last inequality we used the fact that the directed path has at most one contact point with the set $I_+^\alpha$ and, therefore,
 $\sum_{(n,x)\in I_+^\alpha} 1_{x+x_{m-n}=y} \leq 1$. Hence
 \begin{eqnarray*}
 r_N^+\leq \frac{1}{| I_+^\alpha |^{1/4}} e^{\frac{1}{4}(\lambda(-4\beta)+4\lambda(\beta))}.
 \end{eqnarray*}
 The estimate on $s_N^+$ follows along the same lines. Specifically, we have using \eqref{plus} that
 \begin{align} \label{summoment}
& \mathbb{E}\left[\left(W_{N,\omega,+}\right)^2\right] = \mathbb{E}\left[\left( \sum_{(m,y)\in \mathbb{N}\times \mathbb{Z}^d} W_{N,\omega,+}^{(m,y)} \right)^2\right]\nonumber\\
 &\leq
 \mathbb{E}\left[ \left( \sum_{(m,y)\in \mathbb{N}\times \mathbb{Z}^d} \frac{1}{| I_+^\alpha |} \sum_{(n,x)\in I_+^\alpha}  \mu_{N,\omega}^{(n,x)}\left( 1_{x+x_{m-n}=y}
 \right)  \int_{\tilde\omega_{m,y}\geq\omega_{m,y}} 
  e^{\beta\big(\tilde\omega_{m,y}-\omega_{m,y}\big)} \,d\mathbb{P}(\tilde\omega_{m,y})  \right)^2\right]\nonumber\\
  &\leq
  e^{2\lambda(\beta)} \mathbb{E}\left[ \left( \sum_{(m,y)\in \mathbb{N}\times \mathbb{Z}^d} \frac{1}{| I_+^\alpha |} \sum_{(n,x)\in I_+^\alpha}  \mu_{N,\omega}^{(n,x)}\left( 1_{x+x_{m-n}=y}
 \right) e^{-\beta\omega_{m,y}}
  \right)^2\right]\nonumber\\
  &\leq
   e^{2\lambda(\beta)} \mathbb{E}\left[ \frac{1}{| I_+^\alpha |} \sum_{(n,x)\in I_+^\alpha}  
                    \left( \sum_{(m,y)\in \mathbb{N}\times \mathbb{Z}^d}  \mu_{N,\omega}^{(n,x)}\left( 1_{x+x_{m-n}=y}  \right) e^{-\beta\omega_{m,y}}
  \right)^2\right]\nonumber\\
  &\leq
    e^{2\lambda(\beta)} \mathbb{E}\left[ \frac{1}{| I_+^\alpha |} \sum_{(n,x)\in I_+^\alpha}  
                    \left( \sum_{(m,y)\in \mathbb{N}\times \mathbb{Z}^d}  \mu_{N,\omega}^{(n,x)}\left( 1_{x+x_{m-n}=y}  \right)
  \right)
  \left( \sum_{(m,y)\in \mathbb{N}\times \mathbb{Z}^d}  \mu_{N,\omega}^{(n,x)}\left( 1_{x+x_{m-n}=y}  \right) e^{-2\beta\omega_{m,y}}
  \right)\right]\nonumber\\
  &=
  N\,e^{2\lambda(\beta)} \,
  \frac{1}{| I_+^\alpha |} \sum_{(n,x)\in I_+^\alpha}  
   \sum_{(m,y)\in \mathbb{N}\times \mathbb{Z}^d}  \mathbb{E}\left[ \mu_{N,\omega}^{(n,x)}\left( 1_{x+x_{m-n}=y}  \right) e^{-2\beta\omega_{m,y}} \right]  \nonumber\\
   &\leq
     N\,e^{2\lambda(\beta)} \,
  \frac{1}{| I_+^\alpha |} \sum_{(n,x)\in I_+^\alpha}  
   \sum_{(m,y)\in \mathbb{N}\times \mathbb{Z}^d}  \mathbb{E}\left[ \mu_{N,\omega}^{(n,x)}\left( 1_{x+x_{m-n}=y}  \right) \right] \mathbb{E}\left[ e^{-2\beta\omega_{m,y}}\right] \\
   &=  N^2\,e^{\lambda(-2\beta)+2\lambda(\beta)}\nonumber
 \end{align}
 where 
 in the equalities we used the fact that 
  \begin{equation}\label{rangeN}
    \sum_{(m,y)\in \mathbb{N}\times \mathbb{Z}^d}\mu_{N,\omega}^{(n,x)}\left( 1_{x+x_{m-n}=y} \right) =N,
 \end{equation}
and in the last inequality we used the easily verified fact that $\mu_{N,\omega}^{(n,x)}\left( 1_{x+x_{m-n}=y}\right)$ and $e^{-\beta\omega_{m,y}}$ are negatively correlated.
It follows from \eqref{summoment} that
 \begin{eqnarray*}
 s_N^+\leq N  e^{\frac{1}{2}(\lambda(-2\beta)+2\lambda(\beta))}.
 \end{eqnarray*}
 We now need to show how these estimates extend to $r_N^-,s_N^-$. Using \eqref{posnegpart} and the second equality in \eqref{logdiff2},
 \begin{eqnarray} \label{Fminus}
 &&\left(\overline{F}_{N,(\hat\omega^{(m,y)},\tilde\omega_{m,y})}^{I_+^\alpha}
   - \overline{F}_{N,\omega}^{I_+^\alpha} \right)_- \notag\\
 &\leq&
 -\frac{1}{| I_+^\alpha|}\sum_{(n,x)\in I_+^\alpha}\log \mu_{N,\omega}^{(n,x)}\left( e^{\beta(\tilde\omega_{m,y}-\omega_{m,y})1_{x+x_{m-n=y}}} \right) 1_{\tilde\omega_{m,y}<\omega_{m,y}}.
 \end{eqnarray}
 By Jensen's inequality this is bounded by
 \begin{eqnarray} \label{jensen}
&& \frac{1}{| I_+^\alpha|}
  \sum_{(n,x)\in I_+^\alpha}  \mu_{N,\omega}^{(n,x)}
 \left( 1_{x+x_{m-n=y} } \right) 
  \beta(\omega_{m,y}-\tilde\omega_{m,y} )1_{\tilde\omega_{m,y}<\omega_{m,y}}.
 \end{eqnarray}
 It follows that 
 \[
   \mathbb{E}\left[\big(W_{N,\omega,-}^{(m,y)} \big)^2\right] \leq
  \mathbb{E}\left[ \left( \frac{1}{| I_+^\alpha |} \sum_{(n,x)\in I_+^\alpha}  \mu_{N,\omega}^{(n,x)}\left( 1_{x+x_{m-n}=y}
   \right)  \int_{\tilde\omega_{m,y}<\omega_{m,y}} 
  \beta\big(\omega_{m,y}-\tilde\omega_{m,y}\big) \,d\mathbb{P}(\tilde\omega_{m,y})  \right)^2\right]
 \]
 From this we can proceed analogously to \eqref{rNplus} and obtain
 \[
   r_N^- \leq \frac{\beta}{| I_+^\alpha |^{1/4}} \left( \mathbb{E}[\omega_{m,y}^4]  +\mathbb{E}[|\omega_{m,y}|]^4 \right)^{1/4}.
 \]
To bound $s_N^-$ we first observe that
\begin{equation} \label{omegas}
  \int_{\tilde\omega_{m,y}\leq\omega_{m,y}} 
    \big(\omega_{m,y}-\tilde\omega_{m,y}\big) \,d\mathbb{P}(\tilde\omega_{m,y}) \leq (\omega_{m,y})_+ + \E[(\omega_{0,0})_-].
\end{equation}
Using \eqref{posnegpart}, \eqref{rangeN}, \eqref{omegas} and the three equalities in \eqref{logdiff2}, it follows that
\begin{align} \label{summinus}
\mathbb{E}[(W_{N,\omega,-})^2] &= \mathbb{E}\left[\left(\sum_{(m,y)\in \mathbb{N}\times\mathbb{Z}^d} W_{N,\omega,-}^{(m,y)}\right)^2\right]\notag\\
&\leq
\mathbb{E}\left[ \left( \sum_{(m,y)\in \mathbb{N}\times \mathbb{Z}^d} \frac{1}{| I_+^\alpha |} 
  \sum_{(n,x)\in I_+^\alpha}  \mu_{N,\omega}^{(n,x)}\left( 1_{x+x_{m-n}=y}
  \right) \beta ( (\omega_{m,y})_+ + \E[(\omega_{0,0})_-] )  \right)^2\right]\notag\\
&\leq 2\beta^2 N^2 ( \E[(\omega_{0,0})_-] )^2 + 2\beta^2 \mathbb{E}\left[ \left( 
  \frac{1}{| I_+^\alpha |} \sum_{(n,x)\in I_+^\alpha} \mathcal{M}_{N,\omega}^{(n,x)}\right)^2 \right] \notag\\
&\leq 2\beta^2 N^2 ( \E[(\omega_{0,0})_-] )^2 + 2\beta^2 \mathbb{E}\left[( \mathcal{M}_{N,\omega})^2 \right],
\end{align}
where 
 $\mathcal{M}_{N,\omega}$ is from \eqref{absmax}.
A similar computation to the one following \eqref{termtwo} shows that for $L,b$ as chosen after \eqref{bbterm}, with $b$ sufficiently large (depending on $\nu$),
\begin{align} \label{2ndmoment}
  \E[(\mathcal{M}_{N,\omega})^2] &\leq (bN)^2 + \int_{(bN)^2}^\infty \P\left((\mathcal{M}_{N,\omega})^2 > t\right)\ dt \notag \\
  &\leq (bN)^2 + N^2 \int_{b^2}^\infty \P\left(\mathcal{M}_{N,\omega} > N\sqrt{y}\right)\ dy \notag \\
  &\leq (bN)^2 + N^2(2d)^N \int_{b^2}^\infty e^{-N\mathcal{J}(\sqrt{y})}\ dy \notag \\
  &\leq (bN)^2 + N^2(2d)^N \int_{b^2}^\infty e^{-NL\sqrt{y}}\ dy \notag \\
  &\leq (bN)^2 + N^2e^{-LbN/2} \notag\\
  &\leq (b^2+1)N^2.
\end{align}
With \eqref{summinus} this shows that 
\[
  s_N^- \leq K_3 N.
\]

Turning to $\hr_N^\pm(z)$ and $\hs_N^\pm(z)$,
as in \eqref{logdiff2} we have
\begin{align} \label{logdiff3}
  \log Z _{N,(\hat\omega^{(m,y)},\tilde\omega_{m,y})}(z) - \log Z_{N,\omega}(z) 
  &\leq \mu_{N,\omega,z}\left(1_{x_m=y}\right) e^{\beta\big(\tilde\omega_{m,y}-\omega_{m,y}\big)}
\end{align}
and then as in \eqref{rNplus},
\begin{align} \label{rNplus2}
  \mathbb{E}\left[ \big(L_{N,\omega,+}^{(m,y)}(z) \big)^2 \right] &\leq
    \mathbb{E}\left[ \left( \int_{\tilde\omega_{m,y}\geq\omega_{m,y}} 
    e^{\beta\big(\tilde\omega_{m,y}-\omega_{m,y}\big)} \,d\mathbb{P}(\tilde\omega_{m,y})  \right)^2\right]
    \nonumber \\
  &\leq e^{\lambda(-2\beta)+2\lambda(\beta)},
\end{align}
so also
\begin{equation} \label{rNpbound2}
  \hr_N^+(z) \leq e^{\lambda(-2\beta)+2\lambda(\beta)}.
\end{equation}
Further, analogously to \eqref{summoment} but with $I_+^\alpha$ replaced by a single point, we obtain
\begin{align} \label{summoment2}
  \hs_N^+(z)^2 &= \mathbb{E}\left[ \big(L_{N,\omega,+}(z) \big)^2 \right] \notag \\
  &\leq
     N\,e^{2\lambda(\beta)} \, 
     \sum_{(m,y)\in \mathbb{N}\times \mathbb{Z}^d}  
     \mathbb{E}\left[ \mu_{N,\omega,z}\left( 1_{x_m=y}  \right) \right] 
     \mathbb{E}\left[ e^{-2\beta\omega_{m,y}}\right] \notag\\
   &=  N^2\,e^{\lambda(-2\beta)+2\lambda(\beta)}.
\end{align}
Next, analogously to \eqref{Fminus} and \eqref{jensen},
\begin{align} \label{minusbound}
  &\left( \log Z _{N,(\hat\omega^{(m,y)},\tilde\omega_{m,y})}(z) - \log Z_{N,\omega}(z) \right)_- 
    \leq \beta\mu_{N,\omega,z}\left( 1_{x_m=y}  \right) (\omega_{m,y}-\tilde\omega_{m,y} )1_{\tilde\omega_{m,y}<\omega_{m,y}}
\end{align}
so
\begin{equation}
  \mathbb{E}\left[ \big(L_{N,\omega,-}^{(m,y)}(z) \big)^2 \right] \leq 2\beta^2 \E(\omega_{m,y}^2)
\end{equation}
and hence 
\[
  \hr_N^-(z) \leq 2\beta \E(\omega_{0,0}^2)^{1/2}.
\]
To deal with $\hs_N^-(z)$, observe that by \eqref{omegas} and \eqref{minusbound}, similarly to \eqref{summinus},
\begin{align}\label{Lbound}
  L_{N,\omega,-}(z) &\leq \sum_{(m,y)\in \mathbb{N}\times \mathbb{Z}^d} \beta\mu_{N,\omega,z}\left( 1_{x_m=y}  \right)
    ( (\omega_{m,y})_+ + \E[(\omega_{0,0})_-] ) \notag \\
  &\leq \beta \mathcal{M}_{N,\omega} + \beta N \E[(\omega_{0,0})_-].
\end{align}
Therefore $\hs_N^-(z)^2$ is bounded by the right side of \eqref{summinus}, which with \eqref{2ndmoment} shows $\hs_N^-(z) \leq K_3 N$.
\end{proof}

Proposition \ref{rNsN} shows that $\log [N/ (r_Ns_N \log (N / r_Ns_N))]$ is of order $\log N$.  We can apply Proposition \ref{Poinc} and Theorem \ref{main concentration}, the latter with $\rho_N = r_N, \sigma_N = s_N, F= \overline{F}_{N,\omega}^{I_+^\alpha}$ and $K$ a multiple of $N$, to yield part (i) of the next proposition.  Part (ii) follows similarly, using $\hr_N(z)$ and $\hs_N(z)$ in place of $r_N$ and $s_N$, and $F(\omega) = \log Z_{N,\omega}(z)$.  

\begin{proposition}\label{avgconc}
(i) There exist $K_4$ and $N_0=N_0(\beta,\nu)$ such that
\begin{eqnarray*}
\mathbb{P}\left( \left| \overline{F}_{N,\omega}^{I_+^\alpha}-\mathbb{E} \overline{F}_{N,\omega}^{I_+^\alpha}\right|>t\sqrt{\ell(N)}\right)\leq 8e^{-K_4 t},
\end{eqnarray*}
for $t>0$ and $N \geq N_0$, where $\ell(N)=N/\log N$.

(ii) There exists $K_5$ and $N_1=N_1(\beta,\nu)$ such that 
\begin{equation}
\mathbb{P}\left( \left| \log Z_{N,\omega}(z)-\mathbb{E}\log Z_{N,\omega}(z)\right|>t\sqrt{N}\right)\leq 8e^{-K_5t},
\end{equation}
for all $N\geq N_1,t>1$ and all $z \in \ZZ^d$ with $|z|_1 \leq N$.
\end{proposition}


We can now prove the first main theorem.

\begin{proof}[Proof of Theorem \ref{expconc}.]
We start by obtaining an $a.s.$ upper and lower bound on $\log Z_{N,\omega}$. Loosely, for the lower bound we consider a point $(\lfloor N^\alpha \rfloor,x)\in I_+^\alpha$ and we force the polymer started at $(0,0)$ to pass through that point; the energy accumulated by the first part of the polymer, i.e. $\sum_{i=1}^{\lfloor N^\alpha \rfloor} \omega_{i,x_i} $, is then bounded below by the minimum energy that the polymer could accumulate during its first $\lfloor N^\alpha \rfloor$ steps. More precisely, we define
\begin{eqnarray*}
\text{M}_{N,\omega}^{n_1,n_2}:=\max\{| \omega_{n,x} | \colon n_1\leq n \leq n_2,\, |x|_\infty \leq N\},
\end{eqnarray*}
and then bound below by the minimum possible energy:
\begin{eqnarray*}
\sum_{i=1}^{\lfloor N^\alpha \rfloor} \omega_{i,x_i}\geq -N^\alpha \M_{N^\alpha,\omega}^{0,N^{\alpha}}.
\end{eqnarray*}
Letting
\begin{eqnarray*}
\M^+_{N,\omega}:=N^{\alpha}\log (2d)+\beta N^{\alpha} \left( \M_{N^\alpha,\omega}^{0,N^{\alpha}}+ \M_{N+N^\alpha,\omega}^{N,N+N^{\alpha}}\right)
\end{eqnarray*}
we then get that
\begin{eqnarray}\label{l bound 1}
\log Z_{N,\omega}&\geq & \log E\left[ e^{\beta\sum_{i=\lfloor N^\alpha \rfloor+1}^N \omega_{i,x_i} }\big|\ X_{\lfloor N^\alpha \rfloor}=x \right] +\log P(X_{\lfloor N^\alpha \rfloor}=x) -\beta N^{\alpha} \,\M_{N^{\alpha},\omega}^{0,N^{\alpha}}\nonumber\\
&\geq& \log Z_{N,\omega}^{(\lfloor N^\alpha \rfloor,x)} - \M^+_{N,\omega}.
\end{eqnarray}
Averaging \eqref{l bound 1} over $x\in I_+^\alpha$ yields
\begin{eqnarray}\label{lower bound}
\log Z_{N,\omega}&\geq &
\overline{F}_{N,\omega}^{I_+^\alpha}
- \M_{N,\omega}^+.
\end{eqnarray} 
In a related fashion we can obtain an upper bound on $\log Z_{N,\omega}$. In this case we start the polymer from a location $(-\lfloor N^\alpha \rfloor,x)\in I_-^\alpha$ and we force it to pass through $(0,0)$. Letting
\begin{eqnarray}
\M_{N,\omega}^-:=N^\alpha \log(2d)+\beta N^\alpha \left( \M_{N^{\alpha},\omega}^{-N^\alpha,0} +\M_{N,\omega}^{N-N^\alpha,N}\right),
\end{eqnarray}
we then have analogously to \eqref{l bound 1} that
\begin{eqnarray}\label{u bound 1}
\log Z_{N,\omega}^{(-\lfloor N^\alpha \rfloor,x)}\geq \log Z_{N,\omega}-\M_{N,\omega}^-,
\end{eqnarray}
so that, averaging over $I_-^\alpha$,
\begin{eqnarray}\label{upper bound}
\log Z_{N,\omega}&\leq &
\overline{F}_{N,\omega}^{I_-^\alpha}
+ \M_{N,\omega}^-.
\end{eqnarray}
Using the fact that $\overline{F}_{N,\omega}^{I_+^\alpha}$ and $\mathbb{E} \overline{F}_{N,\omega}^{I_-^\alpha}$ have the same distribution, and $\mathbb{E}\log Z_{N,\omega}=\mathbb{E} \overline{F}_{N,\omega}^{I_+^\alpha}=\mathbb{E} \overline{F}_{N,\omega}^{I_-^\alpha}$ we obtain from \eqref{lower bound} and \eqref{upper bound} that
\begin{align}\label{interpolation}
\mathbb{P}&\left( \left| \log Z_{N,\omega}-\mathbb{E}\log Z_{N,\omega}\right|>t\sqrt{\ell(N)}\right)\\
&\leq \mathbb{P}\left(  \overline{F}_{N,\omega}^{I_-^\alpha} -\mathbb{E}\overline{F}_{N,\omega}^{I_-^\alpha}+M_{N,\omega}^- >t\sqrt{\ell(N)}\right)
 +
\mathbb{P}\left( \overline{F}_{N,\omega}^{I_+^\alpha}-\mathbb{E}\overline{F}_{N,\omega}^{I_+^\alpha} -M_{N,\omega}^+<-t\sqrt{\ell(N)}\right) \nonumber\\
&\leq \mathbb{P}\left(  \overline{F}_{N,\omega}^{I_-^\alpha} -\mathbb{E}\overline{F}_{N,\omega}^{I_-^\alpha} >\frac{1}{2}t\sqrt{\ell(N)}\right)
 +
 \mathbb{P}\left( \overline{F}_{N,\omega}^{I_+^\alpha}-\mathbb{E}\overline{F}_{N,\omega}^{I_+^\alpha}<-\frac{1}{2}t\sqrt{\ell(N)}\right)\nonumber\\
&\qquad + \mathbb{P}\left( M_{N,\omega}^+>\frac{1}{2}t\sqrt{\ell(N)}\right) +
  \mathbb{P}\left( M_{N,\omega}^->\frac{1}{2}t\sqrt{\ell(N)}\right)\nonumber \\
&=  \mathbb{P}\left( |\overline{F}_{N,\omega}^{I_+^\alpha}-\mathbb{E}\overline{F}_{N,\omega}^{I_+^\alpha}|>\frac{1}{2}t\sqrt{\ell(N)}\right)
 + \mathbb{P}\left( M_{N,\omega}^+>\frac{1}{2}t\sqrt{\ell(N)}\right) 
 +\mathbb{P}\left( M_{N,\omega}^->\frac{1}{2}t\sqrt{\ell(N)}\right).\nonumber
\end{align}
For $N\geq N_0(\beta,\nu)$, Proposition \ref{avgconc}(i) guarantees that the first term on the right side in \eqref{interpolation} is bounded by $8e^{-K_4t/2}$. The second and the third terms are similar so we consider only the second one. If $t>1$, then for some $K_6$, for large $N$,
\begin{eqnarray*}
\mathbb{P}\left( M_{N,\omega}^+>\frac{1}{2}t\sqrt{\ell(N)}\right)
&\leq&
K_6N^{1+\alpha} \mathbb{P}\left(|\omega_{0,0}|>\frac{t}{8\beta} N^{-\alpha}\sqrt{\frac{N}{\log N}}\right)\\
&\leq&K_6N^{1+\alpha}\exp\left(-\frac{t}{8} \frac{N^{\frac{1}{2}-\alpha}}{\sqrt{\log N}}\right) 
\mathbb{E}\left[e^{\beta|\omega|}\right]\\
&\leq& \exp\left(-\frac{t}{16} \frac{N^{\frac{1}{2}-\alpha}}{\sqrt{\log N}}\right) .
\end{eqnarray*}
 Putting the estimates together we get from \eqref{interpolation} that for some $K_7$,
\begin{eqnarray} \label{three_est}
\mathbb{P}\left( \left| \log Z_{N,\omega}-\mathbb{E}\log Z_{N,\omega}\right|>t\sqrt{\ell(N)}\right)
&\leq&  10e^{-K_7t}
\end{eqnarray}
for all $N$ large (say $N\geq N_2(\beta,\nu)\geq N_0(\beta,\nu)$) and $t>1$.   For $t\leq 1$, \eqref{three_est} is trivially true if we take $K_7$ small enough.  This completes the proof for $N\geq N_2$.  

For $2 \leq N<N_2$ an essentially trivial proof suffices.  Fix any (nonrandom) path $(y_n)_{n\leq N}$ and let $T_N = \sum_{n=1}^N \omega_{n,y_n}$, so that $Z_{N,\omega} \geq (2d)^{-N}e^{\beta T_N}$.  Let $K_8 = N_2\log 2d + \max_{N<N_2} \E\log Z_{N,\omega}$, $K_9 = \min_{N<N_2} \E\log Z_{N,\omega}$ and $K_{10} = \max_{N<N_2} \E Z_{N,\omega}$. Then for some $K_{11},K_{12}$,
\[
  \mathbb{P}\left( \log Z_{N,\omega}-\mathbb{E}\log Z_{N,\omega} < -t\sqrt{\frac{N}{\log N}}\right) \leq \mathbb{P}(\beta T_N <K_8-t)
    \leq K_{11}e^{-K_{12}t}
\]
and by Markov's inequality,
\[
  \mathbb{P}\left( \log Z_{N,\omega}-\mathbb{E}\log Z_{N,\omega} > t\sqrt{\frac{N}{\log N}}\right) \leq \mathbb{P}(Z_{N,\omega}>e^{K_9 + t})
    \leq K_{10}e^{-K_9-t}.
\]
The theorem now follows for these $N\geq 2$.
\end{proof}

\section{Subgaussian rates of convergence}\label{proofrates}

In this section we prove Theorem \ref{ratemain}.
We start with the simple observation that $\E\log Z_{N,\omega}$ is superadditive:
\begin{equation} \label{superadd}
  \E\log Z_{N+M,\omega} \geq \E\log Z_{N,\omega} + \E\log Z_{M,\omega},
\end{equation}
which by standard superadditivity results implies that the limit in \eqref{pbeta} exists, with
\begin{equation} \label{frombelow}
  \lim_{N\to\infty} \frac{1}{N}\E\log Z_{N,\omega} = \sup_N \frac{1}{N}\E\log Z_{N,\omega}.
\end{equation}

  Let $\bbL^{d+1}$ be the even sublattice of $\ZZ^{d+1}$:
\[
  \bbL^{d+1} = \{(n,x) \in \ZZ^{d+1}: n+x_1 + \dots + x_d \text{ is even}\}.
  \]
Let $H_N = \{(N,x): x \in \ZZ^d\} \cap \bbL^{d+1}$ and for $l<m$ and $(l,x),(m,y) \in \bbL^{d+1}$ define
\begin{eqnarray*}
Z_{m-l,\omega}((l,x)(m,y))=E_{l,x}\left[e^{\beta\sum_{n=l+1}^m\omega_{n,x_n}};x_m=y\right].
\end{eqnarray*}
 Recall the notation \eqref{ZN} for a polymer in a shifted disorder. 

The following lemma will be used throughout. Its proof follows the same lines as  (\cite{Al11}, Lemma 2.2(i)) and analogously to that one it is a consequence of Theorem \ref{expconc} for part (i), and Proposition \ref{avgconc}(ii) for part (ii).

\begin{lemma} \label{sums}
Let $\nu$ be nearly gamma.  There exists $K_{13}$ as follows.  Let $n_{\max} \geq 1$ and let $0 \leq s_1<t_1 \leq s_2 < t_2 < \dots \leq s_r < t_r$ with $t_j-s_j \leq n_{\max}$ for all $j \leq r$.  For each $j \leq r$ let $(s_j,y_{j}) \in H_{s_j}$ and $(t_j,z_{j}) \in H_{t_j}$, and let 
\[
  \zeta_j = \log Z_{t_j-s_j,\omega}((s_j,x_j)(t_j,y_j)), \qquad \chi_j = \log Z_{t_j-s_j,\omega}^{(s_j,x_j)}.
  \]
Then for $a>0$, we have the following.

(i)
\begin{equation}
  \P\left( \sum_{j=1}^r |\chi_j - \mathbb{E}\chi_j| > 2a\right) \leq 2^{r+1} \exp\left( -K_{13}a\left( \frac{ \log n_{\max}}{n_{\max}} \right)^{1/2} \right),
  \end{equation}
  
(ii)
\begin{equation} \label{weak2side}
  \P\left( \sum_{j=1}^r |\zeta_j - \mathbb{E}\zeta_j| > 2a\right) \leq 2^{r+1} \exp\left( -\frac{K_{13}a}{n_{\max}^{1/2}} \right),
  \end{equation}
  
(iii)
\begin{equation} \label{oneside}
  \P\left( \sum_{j=1}^r (\zeta_j - \mathbb{E}\chi_j)_+ > 2a\right) \leq 2^{r+1} \exp\left( -K_{13}a\left( \frac{ \log n_{\max}}{n_{\max}} \right)^{1/2} \right),
  \end{equation}
\end{lemma} 

Note (iii) follows from (i), since $\zeta_j \leq \chi_j$.  We do not have a bound like \eqref{oneside}, with factor $(\log n_{\max})^{1/2}$, for the lower tail of the $\zeta_j$'s, but for our purposes such a bound is only needed for the upper tail, as \eqref{weak2side} suffices for lower tails.

We continue with a result which is like Theorem \ref{ratemain} but weaker (not subgaussian) and much simpler. Define the set of paths from the origin
\[
  \Gamma_N = \{ \{(i,x_i)\}_{i\leq N}:  x_0=0,|x_i-x_{i-1}|_1=1 \text{ for all } i\}.
\]
For a specified block length $n$, and for $N=kn$, the \emph{simple skeleton} of a path in $\Gamma_N$ is $\{(jn,x_{jn}):0 \leq j \leq k\}$.  Let $\mC_s$ denote the class of all possible simple skeletons of paths from $(0,0)$ to $(kn,0)$ and note that
\begin{equation} \label{Cs}
  |\mC_s| \leq (2n)^{dk}.
\end{equation}
For a skeleton $\mS$ (of any type, including simple and types to be introduced below), we write $\Gamma_N(\mS)$ for the set of all paths in $\Gamma_N$ which pass through all points of $\mS$.  For a set $\mA$ of paths of length $N$ we set
\[
  Z_{N,\omega}(\mA) = E\left( e^{\beta\sum_{i=1}^N \omega_{i,x_i}} 1_{\mA} \right),
\]
and we write $Z_{N,\omega}(\mS)$ for $Z_{N,\omega}(\Gamma_N(\mS))$.

\begin{lemma}\label{rateweak} 
  Suppose $\nu$ is nearly gamma.  Then there exists $K_{14}$ such that
  \begin{equation} \label{rate3}
    \E\log Z_{n,\omega} \geq p(\beta)n - K_{14}n^{1/2}\log n \quad \text{for all } n \geq 2.
  \end{equation}
\end{lemma}
\begin{proof}
It is sufficient to prove the inequality in \eqref{rate3} for sufficiently large $n$.
Fix $n$ and consider paths of length $N=kn$.  For each $\mS = \{(jn,x_{jn}):0 \leq j \leq k\} \in \mC_s$ we have 
\begin{equation} \label{blocks}
  \E\log Z_{N,\omega}(\mS) = \sum_{j=1}^k \E\log Z_{n,\omega}\bigg( ((j-1)n,x_{(j-1)n}),(jn,x_{jn}) \bigg) \leq k\E\log Z_{n,\omega}.
\end{equation}
By Lemma \ref{sums}(ii) (note $K_{13}$ is defined there),
\begin{align} \label{eachS}
  \P\bigg( \log Z_{N,\omega}(\mS) - \E\log Z_{N,\omega}(\mS) \geq 16dK_{13}^{-1}kn^{1/2}\log n \bigg) \leq 2^{k+1}e^{-8dk\log n},
\end{align}
so by \eqref{Cs},
\begin{align} \label{allS}
  \P\bigg( &\log Z_{N,\omega}(\mS) - \E\log Z_{N,\omega}(\mS) \geq 16dK_{13}^{-1}kn^{1/2}\log n \text{ for some } \mS \in \mC_s \bigg) 
    \leq e^{-4dk\log n}.
\end{align}
Combining \eqref{Cs},\eqref{blocks} and \eqref{allS} we see that with probability at least  $1-e^{-4dk\log n}$ we have 
\begin{align} \label{ceiling}
  \log Z_{kn,\omega} &= \log \left( \sum_{\mS \in \mC_s} Z_{N,\omega}(\mS) \right) \notag\\
  &\leq dk\log(2n) + k\E\log Z_{n,\omega} + 16dK_{13}^{-1}kn^{1/2}\log n.
\end{align}
But by \eqref{aslimit}, also with probability approaching 1 as $k\to \infty$ (with $n$ fixed), we have
\begin{equation} \label{nlimit}
  \log Z_{kn,\omega} \geq knp(\beta) - k
\end{equation}
which with \eqref{ceiling} shows that
\[
  \E\log Z_{n,\omega} \geq np(\beta) -1 - d\log(2n) - 16dK_{13}^{-1}n^{1/2}\log n.
\]
\end{proof}

The proof of Theorem \ref{ratemain} follows the general outline of the preceding proof.  But to obtain that (stronger) theorem, we need to sometimes use Lemma \ref{sums}(i),(iii) in place of (ii), and use a coarse-graining approximation effectively to reduce the size of \eqref{Cs}, so that we avoid the $\log n$ in the exponent on the right side of \eqref{allS}, and can effectively use $\log\log n$ instead.

For $(n,x)\in\bbL^{d+1}$ let 
\[
  s(n,x) = np(\beta) - \E\log Z_{n,\omega}(x), \qquad s_0(n) = np(\beta) - \E\log Z_{n,\omega}.
\]
so $s(n,x) \geq 0$ by \eqref{frombelow}.  $s(n,x)$ may be viewed as a measure of the inefficiency created when a path makes an increment of $(n,x)$.  As in the proof of Lemma \ref{rateweak}, we consider a polymer of length $N=kn$ for some block length $n$ to be specified and $k \geq 1$.  In general we take $n$ sufficiently large, and then take $k$ large, depending on $n$; we tacitly take $n$ to be even, throughout. In addition to \eqref{superadd} we have the relation
\[
  Z_{n+m,\omega}(x+y) \geq Z_{n,\omega}(x)Z_{m,\omega}^{(n,x)}(y) \qquad \text{for all $x,y,z \in \ZZ^d$ and all } n,m\geq 1,
\]
which implies that $s(\cdot,\cdot)$ is subadditive.  Subadditivity of $s_0$ follows from \eqref{superadd}.

Let
\begin{equation} \label{slowlyvar}
  \rho(m) = \frac{\log \log m}{K_{13}(\log m)^{1/2}}, \quad \theta(m) = (\log m)^{5/2}, \quad \text{and} \quad \varphi(m) = \lfloor (\log m)^3 \rfloor.
  \end{equation}
For our designated block length $n$, for $x \in \ZZ^d$ with $(n,x) \in \bbL^d$, we say the transverse increment $x$ is \emph{inadequate} if $s(n,x) > n^{1/2}\theta(n)$, and \emph{adequate} otherwise.  
Note the dependence on $n$ is suppressed in this terminology.  For general values of $m$, we say $(m,x)$ is \emph{efficient} is $s(m,x) \leq 4n^{1/2}\rho(n)$, and \emph{inefficient} otherwise; again there is a dependence on $n$.  For $m=n$, efficiency is obviously a stronger condition than adequateness.  In fact, to prove Theorem \ref{ratemain} it is sufficient to show that for large $n$, there exists $x$ for which $(n,x)$ is efficient.  

Let
\[
  h_n = \max\{|x|_\infty: x \text{ is adequate}\}.
  \]
(Note we have not established any monotonicity for $s(n,\cdot)$, so some sites $x$ with $|x|_\infty\leq h_n$ may be inadequate.)  We wish to coarse-grain on scale $u_n = 2\lfloor h_n/2\varphi(n) \rfloor$.  A \emph{coarse-grained} (or \emph{CG}) \emph{point} is a point of form $(jn,x_{jn})$ with $j \geq 0$ and $x_{jn} \in u_n\ZZ^d$.
A \emph{coarse-grained} (or \emph{CG}) \emph{skeleton} is a simple skeleton $\{(jn,x_{jn}):0 \leq j \leq k\}$ consisting entirely of CG points.  By a \emph{CG path} we mean a path from $(0,0)$ to $(kn,0)$ for which the simple skeleton is a CG skeleton.  

\begin{remark} \label{strategy}
A rough strategy for the proof of Theorem \ref{ratemain} is as follows; what we actually do is based on this but requires certain modifications.  It is enough to show that for some $K_{15}$, for large $n$, $s(n,x) \leq K_{15}n^{1/2}\rho(n)$ for some $x$.  Suppose to the contrary $s(n,x) > K_{15}n^{1/2}\rho(n)$ for all $x$; this means that for every simple skeleton $\mS$ we have 
\[
  \E\log Z_{kn,\omega}(\mS) \leq knp(\beta) - kK_{15}n^{1/2}\rho(n).
\]
The first step is to use this and Lemma \ref{sums} to show that, if we take $n$ then $k$ large, with high probability 
\[
  \log Z_{kn,\omega}(\hat{\mS}) \leq knp(\beta) - \half kK_{15}n^{1/2}\rho(n) \quad \text{for every CG skeleton } \hat{\mS};
  \]
this makes use of the fact that the number of CG skeletons is much smaller than the number of simple skeletons.  The next step is to show that with high probability, every simple skeleton $\mS$ can be approximated by a CG skeleton $\hat{\mS}$ without changing $\log Z_{kn,\omega}(\mS)$ too much, and therefore 
\[
  \log Z_{kn,\omega}(\mS) \leq knp(\beta) - \frac{1}{4} kn^{1/2}K_{15}\rho(n)  \quad \text{for every simple skeleton } \mS.
  \]
The final step is to sum $Z_{kn,\omega}(\mS)$ over simple skeletons $\mS$ (of which there are at most $(2n)^{dk}$) to obtain 
\[
  \log Z_{kn,\omega} \leq dk \log 2n + knp(\beta) - \frac{1}{4} kn^{1/2}K_{15}\rho(n).
\]
Dividing by $kn$ and letting $k\to\infty$ gives a limit which contradicts \eqref{aslimit}; this shows efficient values $x$ must exist.
\end{remark}

We continue with the proof of Theorem \ref{ratemain}.  Let
\[
  \hat{H}_N = \{x \in \ZZ^d: (N,x) \in H_N, |x|_1 \leq N\};
\]
when $N$ is clear from the context we refer to points $x\in \hat{H}_N$ as \emph{accessible sites}.  Clearly $|\hat{H}_N| \leq (2N)^d$.

\begin{lemma} \label{unexcess}
(i) There exists $K_{16}$ such that for all $n\geq 2$, $s(n,0) \leq K_{16}n^{1/2}\log n$.

(ii) There exists $K_{17}$ such that for $n$ large (depending on $\beta$) and even, if $|x|_1 \leq K_{17} n^{1/2}\theta(n)$ then $x$ is adequate.
\end{lemma}
\begin{proof}
We first prove (i).  It suffices to consider $n$ large.  Let $m=n/2$.  It follows from Proposition \ref{avgconc}(ii) that 
\begin{align} \label{eachx}
  \P&\left( \left| \log Z_{m,\omega}(x) - \E\log Z_{m,\omega}(x) \right| \geq 2dK_5^{-1}m^{1/2}\log m
    \text{ for some } x \in \hat{H}_m \right) \notag \\
  &\qquad \leq (2m)^de^{-2d\log m} \notag \\
  &\qquad \leq \half.
\end{align}
It follows from \eqref{eachx}, Theorem \ref{expconc} and Lemma \ref{rateweak} that with probability at least 1/4, for some accessible site $x$ we have 
\begin{align} \label{hyperfrac}
   \exp\left( \E\log Z_{m,\omega}(x) + 2dK_5^{-1}m^{1/2}\log m \right) &\geq Z_{m,\omega}(x) \notag\\
  &\geq \frac{1}{(2m)^d} Z_{m,\omega} \notag\\
  &\geq \frac{1}{(2m)^d} \exp\left( \E\log Z_{m,\omega} - m^{1/2} \right) \notag\\
  &\geq \exp\left( p(\beta)m - 2K_{14}m^{1/2}\log m \right),
\end{align}
and therefore we have the deterministic statement
\begin{equation}
  \E\log Z_{m,\omega}(x) \geq p(\beta)m - K_{18}m^{1/2}\log m.
\end{equation}
Then by symmetry and subadditivity,
\begin{equation} \label{0unexc}
  s(n,0) \leq s(m,x)+s(m,-x) \leq 2K_{18}n^{1/2}\log n.
\end{equation}

Turning to (ii), let $J=2\lfloor K_{19}n^{1/2}\theta(n) \rfloor$, with $K_{19}$ to be specified.  Analogously to \eqref{eachx} we have using Proposition \ref{avgconc}(ii) that for large $n$,
\begin{align} \label{eachx2}
  \P&\left( \left| \log Z_{n-2J,\omega}(-x) - \E\log Z_{n-2J,\omega}(-x) \right| \geq 2dK_5^{-1}n^{1/2}\log n 
    \text{ for some } x \in \hat{H}_J \right) \notag \\
  &\qquad \leq (2J)^d8 e^{-2d\log n} \notag \\
  &\qquad \leq \frac{1}{4}.
\end{align}
Similarly, also for large $n$,
\begin{align} \label{ratio}
  \P&\left( \log Z_{J,\omega}\big((n-2J,x),(n-J,0)\big) > 2p(\beta)J \text{ for some } x \in \hat{H}_J \right) \notag \\
  &\leq \P\left( \log Z_{J,\omega}\big((n-2J,x),(n-J,0)\big) - \E\log Z_{J,\omega}\big((n-2J,x),(n-J,0)\big) > p(\beta)J \text{ for some } x \in \hat{H}_J \right) \notag \\
  &\leq (2J)^d 8e^{-K_5p(\beta)J^{1/2}} \notag \\
  &<\frac{1}{4}.
\end{align}
Then analogously to \eqref{hyperfrac}, since
\[
  Z_{n-J,\omega}(0) = \sum_{x \in \hH_j} Z_{n-2J,\omega}(x)Z_{J,\omega}\big((n-2J,x),(n-J,0)\big),
\]
by \eqref{0unexc}---\eqref{ratio}, Proposition \ref{avgconc}(ii) and Lemma \ref{rateweak}, with probability at least 1/4, for some $x\in\hat{H}_J$ we have
\begin{align} \label{hyperfrac2}
  \exp&\left( \E\log Z_{n-2J,\omega}(x) + 2dK_5^{-1}n^{1/2}\log n \right) \notag\\
  &\geq Z_{n-2J,\omega}(x) \notag\\
  &\geq Z_{n-2J,\omega}(x)Z_{J,\omega}\big((n-2J,x),(n-J,0)\big)e^{-2p(\beta)J} \notag \\
  &\geq \frac{1}{|\hH_J|} Z_{n-J,\omega}(0)e^{-2p(\beta)J} \notag\\
  &\geq \frac{1}{(2J)^d} \exp\left( \E\log Z_{n-J,\omega}(0) - 2dK_5^{-1}n^{1/2}\log n - 2p(\beta)J \right) \notag\\
  &\geq \exp\left( p(\beta)n - 5p(\beta)K_{19}n^{1/2}\theta(n) \right),
\end{align}
and therefore
\begin{equation} \label{thetagap}
  \E\log Z_{n-2J,\omega}(x) \geq p(\beta)n - 6p(\beta)K_{19}n^{1/2}\theta(n).
\end{equation}
If $|y|_1 \leq J$, then $|y-x|_1 \leq 2J$, so there is a path $\{(i,x_i)\}_{n-2J \leq i \leq n}$ from $(n-2J,x)$ to $(n,y)$.  Therefore using \eqref{thetagap},
bounding $Z_{2J,\omega}\big((n-2J,x),(n,y)\big)$ below by the term corresponding to this single path we obtain
\begin{align}
  \E\log Z_{n,\omega}(y) &\geq \E\log Z_{n-2J,\omega}(x) + \E\log Z_{2J,\omega}\big((n-2J,x),(n,y)\big) \notag \\
  &\geq  \E\log Z_{n-2J,\omega}(x) - 2J\log 2d + \beta\E\sum_{i=n-2J+1}^n \omega_{i,x_i} \notag \\
  &= \E\log Z_{n-2J,\omega}(x) - 2J\log 2d  \notag \\
  &\geq p(\beta)n - K_{19}\big(6p(\beta)+4\log 2d \big)n^{1/2}\theta(n).
\end{align}
Taking $K_{19} = (6p(\beta) + 4 \log 2d )^{-1} $, this shows that $y$ is adequate whenever $|y|_1 \leq J$.
\end{proof}

Observe that for a simple skeleton $\mS = \{(jn,x_{jn}), j \leq k\}$, we have a sum over blocks:
\begin{equation} \label{skelsum}
 \log Z_{N,\omega}(\mS) = \sum_{j=1}^k \log Z_{n,\omega}\bigg(((j-1)n,x_{(j-1)n}),(jn,x_{jn})\bigg).
\end{equation}
The rough strategy outlined in Remark \ref{strategy} involves approximating $Z_{N,\omega}(\mS)$ by $Z_{N,\omega}(\hmS)$, where $\hmS$ is a CG skeleton which approximates the simple skeleton $\mS$; equivalently, we want to replace $x_{(j-1)n},x_{jn}$ in \eqref{skelsum} by CG points.  This may be problematic for some values of $j$ and some paths in $\Gamma_N(\mS)$, however, for three reasons.  First, if we do not restrict the possible increments to satisfy $|x_{jn} - x_{(j-1)n}|_\infty \leq h_n$, there will be too many CG skeletons to sum over.  Second, even when increments satisfy this inequality, there are difficulties if increments are inadequate.  Third, paths which veer to far off course transversally within a block present problems in the approximation by a CG path.  Our methods for dealing with these difficulties principally involve two things:  we do the CG approximation only for ``nice'' blocks, and rather than just CG skeletons, we allow more general sums of the form
\[
  \sum_{j=1}^l \log Z_{\tau_j-\tau_{j-1},\omega}((\tau_{j-1},y_{j}),(\tau_j,z_{j})),
\]
which need not have $y_{j}=z_{j-1}$.  We turn now to the details.

In approximating \eqref{skelsum} we want to in effect only change paths within a distance $n_1 \leq 6dn/\varphi(n)$ (to be specified) of each hyperplane $H_{jn}$.  To this end, given a site $w=(jn\pm n_1,y_{jn\pm n_1}) \in H_{jn\pm n_1}$, let $z_{jn}$ be the site in $u_n\ZZ^d$ closest to $y_{jn\pm n_1}$ in $\ell^1$ norm (breaking ties by some arbitrary rule), and let $\pi_{jn}(w) = (jn,z_{jn})$, which may be viewed as the projection into $H_{jn}$ of the CG approximation to $w$ within the hyperplane $H_{jn\pm n_1}$.  Given a path $\gamma=\{(i,x_{i}), i \leq kn\}$ from $(0,0)$ to $(kn,0)$, define points
\[
  d_j = d_j(\gamma) = (jn,x_{jn}), \quad 0 \leq j \leq k,
  \]
\[
  e_j = (jn+n_1,x_{jn+n_1}), \quad 0 \leq j \leq k-1,
  \]
\[
  f_j = (jn-n_1,x_{jn-n_1}), \quad 1 \leq j \leq k.
  \]
We say a \emph{sidestep} occurs in block $j$ in $\gamma$ if either
\[
  |x_{(j-1)n+n_1} - x_{(j-1)n}|_\infty > h_n \quad \text{or} \quad |x_{jn} - x_{jn-n_1}|_\infty > h_n.
  \]
Let 
\[
  \mE_{in} = \mE_{in}(\gamma) = \{1 \leq j \leq k:  x_{jn} - x_{(j-1)n} \text{ is inadequate}\},
  \]
\[
  \mE_{side} = \mE_{side}(\gamma) = \{1 \leq j \leq k:  j \notin \mE_{in} \text{ and a sidestep occurs in block } j\},
  \]
\[
  \mE = \mE_{in} \cup \mE_{side}
  \]
and let
\[
  e_{j-1}' = \pi_{(j-1)n}(e_{j-1}), \quad  f_j' = \pi_{jn}(f_j), \quad j \notin \mE.
  \]
Blocks with indices in $\mE$ are called \emph{bad blocks}, and $\mE$ is called the \emph{bad set}.  Define the tuples
\begin{equation} \label{Rjdef}
  \mT_j = \mT_j(\gamma) = \begin{cases} (d_{j-1},e_{j-1},f_j,d_j) &\text{if } j \in \mE,\\
    (e_{j-1}',f_j') &\text{if } j \notin \mE, \end{cases}
  \end{equation}
define the \emph{CG-approximate skeleton} of $\gamma$ to be
\[
  S_{CG}(\gamma) = \{ \mT_j: 1 \leq j \leq k\}
  \]
and define the \emph{CG-approximate bad} (respectively \emph{good) skeleton} of $\gamma$ to be
\[
  S_{CG}^{bad}(\gamma) = \{ \mT_j: j \in \mE\}, \qquad S_{CG}^{good}(\gamma) = \{ \mT_j: j \notin \mE\}.
  \]
Note $\mE_{in}(\gamma), \mE_{side}(\gamma), S_{CG}^{bad}(\gamma)$ and $S_{CG}^{good}(\gamma)$ are all functions of $S_{CG}(\gamma)$.  
We refer to the bad set $\mE$ also as the \emph{index set of} $S_{CG}^{bad}(\gamma)$.
Let $\mC_{CG}$ (respectively $\mC_{CG}^{bad}$) denote the class of all possible CG-approximate skeletons (respectively bad skeletons) of paths of length $kn$ starting at $(0,0)$.  For $B \subset \{1,\dots,k\}$ let 
$\mC_{CG}(B)$ denote the class of all CG-approximate skeletons in $\mC_{CG}$ with bad set $B$, and 
analogously, let $\mC_{CG}^{bad}(B)$ denote the class of all possible CG-approximate bad skeletons in $\mC_{CG}^{bad}$ with index set $B$.
Then for $b\leq k$ define
\[
  \mC_{CG}^{bad}(b) = \cup_{B:|B|=b}\ \mC_{CG}^{bad}(B).
\]
The partition function corresponding to a CG-approximate skeleton $\mS_{CG}$ is
\begin{align} \label{tZdef}
  \tZ_{N,\omega}(\mS_{CG}) &= \left( \prod_{j\notin\mE} Z_{n_1,\omega}(e_{j-1}',f_j') \right) \left( \prod_{j\in\mE} Z_{n_1,\omega}(d_{j-1},e_{j-1}) Z_{n-2n_1,\omega}(e_{j-1},f_j) 
    Z_{n_1,\omega}(f_j,d_j) \right).
\end{align}
So that we may consider these two products separately, we denote the first as $\tZ_{N,\omega}(\mS_{CG}^{good})$ and the second as $\tZ_{N,\omega}(\mS_{CG}^{bad})$.

For a CG-approximate skeleton in $\mC_{CG}(B)$, and for $j \notin B$, if $e_{j-1}'=(n(j-1),w), d_{j-1}=(n(j-1),x),f_j'=(nj,y)$ and $d_j=(nj,z)$, we always have
\[
  |w-x|_\infty \leq h_n + \frac{u_n}{2}, \quad |z-y|_\infty \leq h_n + \frac{u_n}{2}.
\]
It follows readily that if $\mT_1,\dots,\mT_{j-1}$ are specified and $j \notin B$, then there are at most $(4h_nu_n^{-1}+3)^{2d} \leq (5\varphi(n))^{2d}$ possible values of $\mT_j$; if $j \in B$ there are at most $(2n)^{4d}$.  
It follows that the number of CG-approximate skeletons satisfies 
\begin{equation} \label{CB}
  |\mC_{CG}(B)| \leq (5\varphi(n))^{2d(k-|B|)}(2n)^{4d|B|}.
  \end{equation}
Note that the factor $\varphi(n)$ in place of $n$ in \eqref{CB} represents the entropy reduction resulting from the use of CG paths. 
Summing \eqref{CB} over $B$ we obtain
\begin{equation} \label{CGsize}
  |\mC_{CG}| \leq 2^k (2n)^{4dk}.
\end{equation}
For $B=\{j_1<\dots<j_{|B|}\} \subset \{1,\dots,k\}$, setting $j_0=0$ we have
\begin{equation} \label{Cbadsize}
  |\mC_{CG}^{bad}(B)| \leq \prod_{1\leq i \leq |B|} [(2(j_i-j_{i-1})n)^d (2n)^{3d}]\leq \left( \frac{16n^4k}{|B|} \right)^{d|B|},
\end{equation}
so for each $b \leq k$, using ${k\choose b} \leq (ke/b)^b$,
\begin{equation}\label{Cbadbsize}
  |\mC_{CG}^{bad}(b)| \leq {k\choose b} \left( \frac{16n^4k}{b} \right)^{db} \leq \left( \frac{8n^2k}{b} \right)^{2db}.
\end{equation}

We also use the non-coarse-grained analogs of the $\mT_j$, given by
\begin{equation} \label{Vjdef}
  \mV_j = \mV_j(\gamma) = (d_{j-1},e_{j-1},f_j,d_j), \quad j \leq k,
\end{equation}
and define the \emph{augmented skeleton} of $\gamma$ to be
\[
  \mS_{aug}(\gamma) = \{\mV_j, 1 \leq j \leq k\}.  
  \]
We write $\mC_{aug}$ for the class of all possible augemented skeletons of paths from $(0,0)$ to $(kn,0)$.
Note that $\mE_{side}(\gamma),\mE_{in}(\gamma)$ and $\mS_{CG}(\gamma)$ are functions of $\mS_{aug}(\gamma)$; we denote by $F$ the ``coarse-graining map'' such that
\[
  \mS_{CG}(\gamma) = F\left( \mS_{aug}(\gamma) \right).
\]
We can write
\[
  Z_{N,\omega} = \sum_{\mS_{CG}\in\mC_{CG}}\ \sum_{\mS_{aug} \in F^{-1}(\mS_{CG})} Z_{N,\omega}(\mS_{aug}),
\]
and define
\[
  \tZ_{N,\omega} = \sum_{\mS_{CG}\in\mC_{CG}}\ |F^{-1}(\mS_{CG})| \tZ_{N,\omega}(\mS_{CG}).
\]
Now for a given choice of $e_{j-1}'$ there are at most $(2n_1)^d$ possible choices of $e_{j-1}$ and then at most $(2n_1)^d$ for $d_{j-1}$, and similarly for $f_j',f_j,d_j$, so for all $\mS_{CG}$,
\begin{equation} \label{Fsize}
  |F^{-1}(\mS_{CG})| \leq (2n_1)^{4dk}.
\end{equation}

The following will be proved in the next section.

\begin{lemma} \label{fastpi}
For $n$ sufficiently large, there exists an even integer $n_1 \leq 6dn/\varphi(n)$ such that for all $p \in H_{n_1}$ we have
\[
  \E\log Z_{n_1,\omega}(\pi_0(p),p) \geq p(\beta)n_1 - 20dn^{1/2}\rho(n).
\]
\end{lemma}

This lemma is central to the following, which bounds the difference between partition functions for a skeleton and for its CG approximation.

\begin{lemma} \label{CGcorrec}
There exists $K_{20}$ such that under the conditions of Theorem \ref{expconc}, for $n$ sufficiently large,
\begin{align} \label{CGcorrec1}
  \P&\left( \log Z_{N,\omega}(\mS_{aug}) - \log \tZ_{N,\omega}(F(\mS_{aug})) \geq 80dkn^{1/2}\rho(n) 
    \text{ for some } \mS_{aug} \in \mC_{aug} \right) \notag \\
  &\leq e^{-K_{20}k(\log n)(\log \log n)}.
\end{align}
\end{lemma}
\begin{proof}
We have
\begin{align} \label{splitup}
  \P&\left( \log Z_{N,\omega}(\mS_{aug}) - \log \tZ_{N,\omega}(F(\mS_{aug})) \geq 80dkn^{1/2}\rho(n) 
    \text{ for some } \mS_{aug} \in \mC_{aug} \right) \notag \\
  &\leq \sum_{\mS_{CG} \in \mC_{CG}}\ \sum_{\mS_{aug} \in F^{-1}(\mS_{CG})}
    \P\left( \log Z_{N,\omega}(\mS_{aug}) - \log \tZ_{N,\omega}(\mS_{CG}) \geq 80dkn^{1/2}\rho(n) \right).
\end{align}
Fix $\mS_{CG} \in \mC_{CG}$ and $\mS_{aug} \in F^{-1}(\mS_{CG})$.  We can write $\mS_{aug}$ as $\{\mV_j,j\leq k\}$ with $\mV_j$ as in \eqref{Vjdef}.
Then using Lemma \ref{rateweak},
\begin{align} \label{sumsmall}
  \log &Z_{N,\omega}(\mS_{aug}) - \log \tZ_{N,\omega}(\mS_{CG}) \notag\\
  &\leq \sum_{j \notin B} \bigg[ \left( \log Z_{n_1,\omega}(d_{j-1},e_{j-1})
    - \log Z_{n_1,\omega}(e_{j-1}',e_{j-1}) \right) \notag \\
  &\qquad \qquad + \left( \log Z_{n_1,\omega}(f_j,d_j) - \log Z_{n_1,\omega}(f_j,f_j') \right) \bigg] \notag\\
  &\leq \sum_{j \notin B} \bigg[ \left( \log Z_{n_1,\omega}(d_{j-1},e_{j-1}) - \E\log Z_{n_1,\omega}(d_{j-1},e_{j-1}) \right) \notag \\
  &\qquad\qquad - \left( \log Z_{n_1,\omega}(e_{j-1}',e_{j-1}) - \E\log Z_{n_1,\omega}(e_{j-1}',e_{j-1}) \right) \notag \\
  &\qquad\qquad + \left( \log Z_{n_1,\omega}(f_j,d_j) - \E\log Z_{n_1,\omega}(f_j,d_j) \right) \notag \\
  &\qquad\qquad - \left( \log Z_{n_1,\omega}(f_j,f_j') - \E\log Z_{n_1,\omega}(f_j,f_j') \right) \bigg] \notag \\
  &\qquad + \sum_{j \notin B} \left[ 2p(\beta) n_1 - \E\log Z_{n_1,\omega}(e_{j-1}',e_{j-1})
    - \E\log Z_{n_1,\omega}(f_j,f_j') \right].
\end{align}
By Lemma \ref{fastpi}, the last sum is bounded by $40dkn^{1/2}\rho(n)$.  Hence letting $T$ denote the first sum on the right side of \eqref{sumsmall}, we have by \eqref{sumsmall} and Lemma \ref{sums}(ii):
\begin{align} \label{logdiff}
  \P&\left( \log Z_{N,\omega}(\mS_{aug}) - \log \tZ_{N,\omega}(\mS_{CG}) \geq 80dkn^{1/2}\rho(n) \right) \notag\\
  &\leq \P\left( T > 40dkn^{1/2}\rho(n) \right) \notag\\
  &\leq 2^{2k+1}\exp\left(- 20K_{13}dk\rho(n)\left( \frac{n}{n_1} \right)^{1/2} \right) \notag \\
  &\leq e^{-kK_{21}(\log n)(\log \log n)}.
\end{align}
Combining \eqref{splitup} and \eqref{logdiff} with \eqref{CGsize} and \eqref{Fsize} we obtain that for large $n$,
\begin{align} \label{combine}
  \P&\left( \log Z_{N,\omega}(\mS_{aug}) - \log \tZ_{N,\omega}(F(\mS_{aug})) \geq 80dkn^{1/2}\rho(n) 
    \text{ for some } \mS_{aug} \in \mC_{aug} \right) \notag\\
  &\leq (2n)^{9dk} e^{-kK_{21}(\log n)(\log \log n)} \notag\\
  &\leq e^{-kK_{21}(\log n)(\log \log n)/2}.
\end{align}
\end{proof}

It is worth noting that in \eqref{combine} we do not make use of the entropy reduction contained in \eqref{CB}.  Nonetheless we are able to obtain a good bound because we apply Lemma \ref{sums}(ii) with $n_{max}=n_1$ instead of $n_{max}=n$.

Let $b_{nk} = \lfloor \frac{k\log\log n}{(\log n)^{3/2}} \rfloor$. We deal separately with CG-approximate skeletons according to whether the number of bad blocks exceeds $b_{nk}$.  Let
\[
  \mC_{CG}^- = \cup_{B:|B|\leq b_{nk}} \mC_{CG}(B), \quad \mC_{CG}^+ = \cup_{B:|B|> b_{nk}} \mC_{CG}(B).
\]

The next lemma shows that bad blocks have a large cost, in the sense of reducing the mean of the log partition function---compare the $n^{1/2}\theta(n)$ factor in \eqref{badcost2} to the $n^{1/2}\log n$ factor in \eqref{rate3}.

\begin{lemma} \label{badcost}
For $n$ sufficiently large, for all $1 \leq b \leq k$ and $\mS_{CG}^{bad} \in \mC_{CG}^{bad}(b)$, 
\begin{equation} \label{badcost2}
  \E\log \tZ_{N,\omega}(\mS_{CG}^{bad}) \leq p(\beta)bn - \half bn^{1/2}\theta(n).
\end{equation}
\end{lemma}
\begin{proof}
Fix $B \subset \{1,\dots,k\}$ with $|B|=b$, fix $\mS_{CG}^{bad} \in \mC_{CG}^{bad}(B)$, let $\mE_{in},\mE_{side}$ be the corresponding sets of indices of bad blocks, and let $\{\mT_j, j \in B\}$ be as in \eqref{Rjdef}.  Then 
\begin{align} \label{CGdecomp}
  \E\log &\tZ_{N,\omega}(\mS_{CG}^{bad}) \notag\\
  &= \sum_{j \in B} \big[ \E\log Z_{n_1,\omega}(d_{j-1},e_{j-1}) + \E\log Z_{n-2n_1,\omega}(e_{j-1},f_j)
    + \E\log Z_{n_1,\omega}(f_j,d_j) \big].
\end{align}
For $j \in \mE_{in}$ we have 
\begin{align} \label{exc}
  \E\log &Z_{n_1,\omega}(d_{j-1},e_{j-1}) + \E\log Z_{n-2n_1,\omega}(e_{j-1},f_j)
    + \E\log Z_{n_1,\omega}(f_j,d_j) \notag\\
  &\leq \E\log Z_{n,\omega}(d_{j-1},d_j) \notag \\
  &\leq p(\beta)n - n^{1/2}\theta(n).
\end{align}
For $j \in \mE_{side}$, write $e_{j-1}-d_{j-1}$ as $(n_1,x)$, so $|x|_\infty > h_n$ and therefore $x$ is inadequate.  If the sidestep occurs from $(j-1)n$ to $(j-1)n+n_1$, then by superadditivity and Lemma \ref{unexcess}(i), 
\begin{align} \label{straight}
  \E\log Z_{n_1,\omega}(d_{j-1},e_{j-1}) &= \E\log Z_{n_1,\omega}((0,0),(n_1,x)) \notag\\
  &\leq \E\log Z_{n,\omega}((0,0),(n,x)) - \E\log Z_{n-n_1,\omega}((n_1,x),(n,x)) \notag \\
  &\leq p(\beta)n - n^{1/2}\theta(n) - \left( p(\beta)(n-n_1) - K_{16}n^{1/2}\log n \right) \notag\\
  &\leq p(\beta)n_1 - \half n^{1/2}\theta(n),
\end{align}
and therefore
\begin{align} \label{side}
  \E\log &Z_{n_1,\omega}(d_{j-1},e_{j-1}) + \E\log Z_{n-2n_1,\omega}(e_{j-1},f_j)
    + \E\log Z_{n_1,\omega}(f_j,d_j) \notag\\
  &\leq p(\beta)n - \half n^{1/2}\theta(n).
\end{align}
Combining \eqref{CGdecomp}, \eqref{exc} and \eqref{side} we obtain
\begin{equation}
  \E\log \tZ_{N,\omega}(\mS_{CG}^{bad}) 
    \leq p(\beta)bn - \half bn^{1/2}\theta(n).
\end{equation}
\end{proof}

It follows by additivity that 
\begin{equation} \label{ecompare}
  \E\log \tZ_{N,\omega}(\mS_{CG}) \leq k\E\log Z_{n,\omega}
\end{equation}
for all CG skeletons $\mS_{CG}$.  Rather than considering deviations of $\log \tZ_{N,\omega}(\mS_{CG})$ above its mean, it will be advantageous to consider deviations above the right side of \eqref{ecompare}.
The next two lemmas show that it is unlikely for this deviation to be very large for any CG skeleton.  We will use the fact that for each $\mS_{CG} \in \mC_{CG}$ with bad set $B$, we have by Lemmas \ref{rateweak} and \ref{badcost}
\begin{align} \label{meandiff}
  |B|\E\log Z_{n,\omega} - \E\log \tZ_{N,\omega}(\mS_{CG}^{bad}) \geq \half |B|n^{1/2}\theta(n) - K_{14}|B|n^{1/2}\log n
    \geq \frac{1}{4}|B|n^{1/2}\theta(n).
\end{align}

\begin{lemma} \label{fewbad}
Under the conditions of Theorem \ref{expconc}, if $n$ and then $k$ are chosen sufficiently large,
\begin{align} \label{CGbound1}
  P&\bigg( \log \tZ_{N,\omega}(\mS_{CG}) - k\E\log Z_{n,\omega} \geq 80dkn^{1/2}\rho(n) 
    \text{ for some } \mS_{CG} \in \mC_{CG}^- \bigg) \notag \\
  &\leq e^{-16K_{13}dk\rho(n)}.
\end{align}
\end{lemma}

\begin{proof}
From \eqref{CB} we see that
\begin{equation} \label{CCGmsize}
  |\mC_{CG}^-| \leq 2^k(5\varphi(n))^{2dk} (2n)^{4db_{nk}} \leq e^{10dk\log\log n}.
\end{equation}
Combining this with Lemma \ref{sums}(ii),(iii) (with $n_{max}=n$) and \eqref{CGsize}, \eqref{Cbadbsize}, we obtain
\begin{align} \label{tZbound}
  \P\bigg( &\log \tZ_{N,\omega}(\mS_{CG}) - k\E\log Z_{n,\omega} \geq 80dkn^{1/2}\rho(n) 
    \text{ for some } \mS_{CG} \in \mC_{CG}^- \bigg) \notag\\
  &\leq \sum_{b=0}^{b_{nk}} \P\bigg( \log \tZ_{N,\omega}(\mS_{CG}^{bad}) - b\E\log Z_{n,\omega} \geq 40dkn^{1/2}\rho(n) 
    \text{ for some } \mS_{CG}^{bad} \in \mC_{CG}^{bad}(b) \bigg)\notag\\
  &\qquad + \P\bigg( \log \tZ_{N,\omega}(\mS_{CG}^{good}) - (k-|B|)\E\log Z_{n,\omega} \geq 40dkn^{1/2}\rho(n) 
    \text{ for some } \mS_{CG} \in \mC_{CG}^- \bigg) \notag\\
  &\leq \sum_{b=0}^{b_{nk}} \P\bigg( \log \tZ_{N,\omega}(\mS_{CG}^{bad}) - \E\log \tZ_{N,\omega}(\mS_{CG}^{bad})
    \geq 40dkn^{1/2}\rho(n) \text{ for some } \mS_{CG}^{bad} \in \mC_{CG}^{bad}(b) \bigg) \notag\\
  &\qquad + |\mC_{CG}^-|2^{k+1} \exp\left( -20K_{13}dk\rho(n)(\log n)^{1/2} \right)  \notag\\
  &\leq \sum_{b=1}^{b_{nk}} |\mC_{CG}^{bad}(b)| 2^{3b+1} e^{-20K_{13}dk\rho(n)}
    + |\mC_{CG}^-|2^{k+1} e^{-20dk\log\log n}  \notag\\
  &\leq \sum_{b=1}^{b_{nk}} \left( \frac{32n^2k}{b} \right)^{2db} e^{-20K_{13}dk\rho(n)}
    + e^{-9dk\log\log n}.
\end{align}
Note that the event in the third line of \eqref{tZbound} is well-defined because $\mS_{CG}^{good}$ is a function of $\mS_{CG}$.
For each $b \leq b_{nk}$ we have
\begin{equation} \label{bnkprop1}
  \frac{ \log \frac{32n^2k}{b} }{ \frac{32n^2k}{b} } \leq \frac{ \log \frac{32n^2k}{b_{nk}} }{ \frac{32n^2k}{b_{nk}} }
    \leq \frac{ 3 \log\log n }{ 32n^2(\log n)^{1/2} }
\end{equation}
so
\begin{equation} \label{bnkprop2}
  2db\log \frac{32n^2k}{b} \leq \frac{ 3dk \log\log n }{ (\log n)^{1/2} } = 3K_{13}dk\rho(n).
\end{equation}
With \eqref{tZbound} this shows that for $k$ sufficiently large (depending on $n$),
\begin{align} \label{tZbound2}
  \P\bigg( &\log \tZ_{N,\omega}(\mS_{CG}) - k\E\log Z_{n,\omega} \geq 80dkn^{1/2}\rho(n) 
    \text{ for some } \mS_{CG} \in \mC_{CG}^- \bigg) \notag\\
  &\leq b_{nk}e^{-17K_{13}dk\rho(n)} + e^{-9dk\log\log n} \notag\\
  &\leq e^{-16K_{13}dk\rho(n)}.
\end{align}
\end{proof}

We continue with a similar but simpler result for $\mC_{CG}^+$.

\begin{lemma} \label{manybad2}
Under the conditions of Theorem \ref{ratemain}, for $n$ sufficiently large and $N=kn$,
\begin{align} \label{CGbound1}
  P&\bigg( \log \tZ_{N,\omega}(\mS_{CG}) - k\E\log Z_{n,\omega} \geq 0
    \text{ for some } \mS_{CG} \in \mC_{CG}^+ \bigg) \leq e^{-K_{13}k(\log n)(\log\log n)/16}.
\end{align}
\end{lemma}

\begin{proof}
In contrast to \eqref{CB}, it is straightforward that
\begin{equation} \label{CCGplus}
  |\mC_{CG}^+| \leq (2n)^{4dk}.
  \end{equation}
Using \eqref{meandiff} we obtain that for $\mS_{CG} \in \mC_{CG}(B)$ with $|B|\geq b_{nk}$,
\begin{align} \label{lowmean}
  k\E&\log Z_{n,\omega} - \E\log \tZ_{N,\omega}(\mS_{CG}) \notag\\
  &= \left[ |B|\E\log Z_{n,\omega} - \E\log \tZ_{N,\omega}(\mS_{CG}^{bad}) \right] 
    + \left[ (k-|B|)\E\log Z_{n,\omega} - \E\log \tZ_{N,\omega}(\mS_{CG}^{good}) \right] \notag\\
  &\geq \frac{1}{4}b_{nk}n^{1/2}\theta(n).
\end{align}
  
Combining this with Lemma \ref{sums}(ii) (with $n_{max}=n$) and \eqref{CCGplus}, we obtain
\begin{align} \label{tZbound0}
  \P\bigg( &\log \tZ_{N,\omega}(\mS_{CG}) - k\E\log Z_{n,\omega} \geq 0 
    \text{ for some } \mS_{CG} \in \mC_{CG}^+ \bigg) \notag\\
  &\leq \P\bigg( \log \tZ_{N,\omega}(\mS_{CG}) - \E\log \tZ_{N,\omega}(\mS_{CG}) \geq \frac{1}{4}b_{nk}n^{1/2}\theta(n) 
    \text{ for some } \mS_{CG} \in \mC_{CG}^+ \bigg) \notag\\
  &\leq |\mC_{CG}^+|2^{3k+1} e^{-K_{13}b_{nk}\theta(n)/8} \notag\\
  &\leq e^{-K_{13}k(\log n)(\log\log n)/16}.
\end{align}
\end{proof}

We can now complete the proof of Theorem \ref{ratemain}.
If we take $n$ and then $k$ large,
with probability greater than 1/2, none of the events given in Lemmas \ref{CGcorrec}, \ref{fewbad} and \ref{manybad2} occur, and we then have for all $\mS_{aug} \in \mC_{aug}$:
\begin{align} \label{eachskel}
  \log Z_{N,\omega}(\mS_{aug}) &\leq k\E\log Z_{n,\omega} + 160dkn^{1/2}\rho(n).
\end{align}
Then since $|\mC_{aug}| \leq (2n)^{3dk}$, summing over $\mS_{aug} \in \mC_{aug}$ shows that, still with probability greater than 1/2,
\begin{equation} \label{allskel}
  \log Z_{kn,\omega} \leq k\E\log Z_{n,\omega} + 160dkn^{1/2}\rho(n) + 3dk\log(2n) \leq k\E\log Z_{n,\omega} + 161dkn^{1/2}\rho(n).
\end{equation}
By \eqref{aslimit}, for fixed $n$, for sufficiently large $k$ we have, again with probability greater than 1/2:
\begin{equation} \label{aslimit2}
  \frac{1}{kn}\log Z_{kn,\omega} \geq p(\beta) - \frac{1}{n}.
\end{equation}
Thus with positive probability, both \eqref{allskel} and \eqref{aslimit2} hold, and hence
\[
  k\E\log Z_{n,\omega} + 161dkn^{1/2}\rho(n) \geq knp(\beta) - k,
\]
which implies
\[
  \E\log Z_{n,\omega} \geq np(\beta) - 162dn^{1/2}\rho(n).
\]

\section{Proof of Lemma \ref{fastpi}}\label{prooflem}

We begin with some definitions.  
A path $((l,x_{l}),(l+1,x_{l+1}),\dots,(l+m,x_{l+m}))$ is \emph{clean} if every increment $(t-s,x_{t}-x_{s})$ with $l \leq s < t \leq l+m$ is efficient.
Let $x^*$ be an adequate site with first coordinate $x_1^*=|x^*|_\infty=h_n$. Given a path $\gamma = \{(m,x_{m})\}$ from $(0,0)$ to $(n,x^*)$, let 
\[
  \tau_j = \tau_j(\gamma) = \min\{m: (x_m)_1 = ju_n\}, \quad 1 \leq j \leq \varphi(n).
  \]
The \emph{climbing skeleton} of $\gamma$ is $\mS_{cl}(\gamma) = \{(\tau_j,x_{\tau_j}): 1 \leq j \leq \varphi(n)\}$.
A \emph{climbing segment} of $\gamma$ is a segment of $\gamma$ from $(\tau_{j-1},x_{\tau_{j-1}})$ to $(\tau_j,x_{\tau_j})$ for some $j$.  A climbing segment is \emph{short} if $\tau_j - \tau_{j-1} \leq 2n/\varphi(n)$, and \emph{long} otherwise.  (Note $n/\varphi(n)$ is the average length of the climbing segments in $\gamma$.)  Since the total length of $\gamma$ is $n$, there can be at most $\varphi(n)/2$ long climbing segments in $\gamma$, so there are at least $\varphi(n)/2$ short ones.  Let
\[
  \mJ_s(\gamma) = \{j \leq \varphi(n): \text{ the $j$th climbing segment of $\gamma$ is short} \},
  \]
\[
  \mJ_l(\gamma) = \{j \leq \varphi(n): \text{ the $j$th climbing segment of $\gamma$ is long} \},
  \]
\[
  J_l(\gamma) = \left( \cup_{j \in \mJ_l(\gamma)} (\tau_{j-1},\tau_j) \right) \cap \left\lfloor \frac{2n}{\varphi(n)} \right\rfloor \ZZ.
  \]
If no short climbing segment of $\gamma$ is clean, we say $\gamma$ is \emph{soiled}.  For soiled $\gamma$, for each $j \in \mJ_s(\gamma)$ there exist $\alpha_j(\gamma)<\beta_j(\gamma)$ in $[\tau_{j-1},\tau_j]$ for which the increment of $\gamma$ from $(\alpha_j,x_{\alpha_j})$ to $(\beta_j,x_{\beta_j})$ is inefficient.  (If $\alpha_j,\beta_j$ are not unique we make a choice by some arbitrary rule.)  
We can reorder the values $\{\tau_j, j \leq \varphi(n)\} \cup \{\alpha_j,\beta_j: j \in \mJ_s(\gamma)\} \cup J_l(\gamma)$ into a single sequence $\{\sigma_j, 1 \leq j \leq N(\gamma)\}$ with $\varphi(n) \leq N(\gamma) \leq 4\varphi(n)$, such that at least $\varphi(n)/2$ of the increments $(\sigma_j - \sigma_{j-1},x_{\sigma_j} - x_{\sigma_{j-1}}), j \leq N(\gamma),$ are inefficient.  The \emph{augmented climbing skeleton} of $\gamma$ is then the sequence 
$\mS_{acl}(\gamma) = \{(\sigma_j,x_{\sigma_j}): 1 \leq j \leq N(\gamma)\}$.  The set of all augmented climbing skeletons of soiled paths from $(0,0)$ to $(n,x^*)$ is denoted $\mC_{acl}$.

\begin{lemma} \label{existence}
Provided $n$ is large, there exists a path from $(0,0)$ to $(n,x^*)$ containing a short climbing segment which is clean.
\end{lemma}

Note that Lemma \ref{existence} is a purely deterministic statement, since the property of being clean does not involve the configuration $\omega$.  

Translating the segment obtained in Lemma \ref{existence} to begin at the origin, we obtain a path $\alpha^*$ from $(0,0)$ to some site $(m^*,y^*)$, with the following properties:
\begin{equation} \label{properties}
  m^* \leq \frac{2n}{\varphi(n)}, \quad y_1^* = u_n \quad \text{and $\alpha^*$ is clean}.
  \end{equation}
By definition, every increment of $\alpha^*$ is efficient.  The proof of (\cite{Al11}, Lemma 2.3) then applies unchanged:  for $n_1=2d(m^*+1)$, given $p \in \hH_{n_1}$ with $\pi_0(p)=0$, one can find $4d+1$ segments of $\alpha^*$ (or reflections of such segments through coordinate hyperplanes, which are necessarily also efficient) such that the sum of the increments made by these segments is $p$.  By subadditivity this shows that $s(p) \leq (4d+1)n^{1/2}\rho(n)$, proving Lemma \ref{fastpi}.
  
\begin{proof}[Proof of Lemma \ref{existence}]
Let $\mD^*$ denote the set of all soiled paths from $(0,0)$ to $(n,x^*)$.  We will show that $\P(Z_{n,\omega}(\mD^*) < Z_{n,\omega}(x^*))>0$, which shows that unsoiled paths exist, proving the lemma.

Since $x^*$ is adequate, it follows from Proposition \ref{avgconc}(ii) that 
\begin{equation} \label{bigprob}
  \P\bigg( \log Z_{n,\omega}(x^*) > p(\beta)n - 2n^{1/2}\theta(n) \bigg) > \half.
\end{equation}
On the other hand, for paths in $\mathcal{D}^*$, fixing $\mS_{acl} = \{(\sigma_j,x_{\sigma_j}): 1 \leq j \leq r\} \in \mC_{acl}$, since there are at least $\varphi(n)/2$ inefficient increments $(\sigma_j - \sigma_{j-1},x_{\sigma_j} - x_{\sigma_{j-1}})$, we have
\begin{equation} \label{ESacl}
  \E \log Z_{n,\omega}(\mS_{acl}) \leq p(\beta)n - 2n^{1/2} \varphi(n)\rho(n).
\end{equation}
Hence by Lemma \ref{sums}(ii) (with $n_{\max} = 6dn/\varphi(n)$),
\begin{align} \label{eachSacl}
  \P&\left( \log Z_{n,\omega}(\mS_{acl}) \geq  p(\beta)n - n^{1/2} \varphi(n)\rho(n) \right) \notag\\
  &\leq \P\left( \log Z_{n,\omega}(\mS_{acl}) - \E\log Z_{n,\omega}(\mS_{acl}) \geq n^{1/2} \varphi(n)\rho(n) \right) \notag\\
  &\leq  2^{4\varphi(n)+1} \exp\left( -\frac{K_{13}}{2} n^{1/2} \varphi(n)\rho(n)\left( \frac{\varphi(n)}{6dn} \right)^{1/2} \right) \notag\\
  &\leq e^{-(\log n)^4(\log\log n)/6d^{1/2}}.
\end{align}
Since $3\theta(n) \leq \varphi(n)\rho(n)$ and
\begin{equation} \label{Caclsize}
  |\mC_{acl}| \leq (2n)^{4(d+1)\varphi(n)},
\end{equation}
it follows from \eqref{eachSacl} that, in contrast to \eqref{bigprob},
\begin{align} \label{allSacl}
  \P&\left( \log Z_{n,\omega}(\mD^*) > p(\beta)n - 2n^{1/2}\theta(n) \right) \notag\\
  &\leq \P\left( \log |\mC_{acl}| + \max_{\mS_{acl} \in \mC_{acl}} \log Z_{n,\omega}(\mS_{acl}) > p(\beta)n - 2n^{1/2}\theta(n) \right) \notag\\
  &\leq \P\left( \log Z_{n,\omega}(\mS_{acl}) \geq  p(\beta)n - 3n^{1/2}\theta(n) \text{ for some } \mS_{acl} \in \mC_{acl} \right) \notag\\
  &\leq |\mC_{acl}| e^{-(\log n)^4(\log\log n)/6d^{1/2}} \notag\\
  &\leq e^{-(\log n)^4(\log\log n)/12d^{1/2}}.
\end{align}
It follows from \eqref{bigprob} and \eqref{allSacl} that $\P(Z_{n,\omega}(\mD^*) < Z_{n,\omega}(x^*))>0$, as desired.
\end{proof}

\begin{remark} \label{exponents}
The exponents on $\log m$ in the definition \eqref{slowlyvar} of $\theta(m)$ and $\varphi(m)$ are not the only ones that can be used.  The proof of Lemma \ref{CGcorrec} requires (ignoring constants) $\varphi(n) \geq (\log n)^3$, Lemma \ref{manybad2} requires $\theta(n) \geq (\log n)^{5/2}$ and Lemma \ref{existence} requires $\varphi(n) \geq \theta(n)(\log n)^{1/2}$.
\end{remark}

\end{document}